\documentclass[a4,11pt,
abstract=on
]{scrartcl}
\usepackage{algorithm, booktabs}
\usepackage{mathtools}
\mathtoolsset{showonlyrefs}
\usepackage{amsmath}
\usepackage{amssymb}
\usepackage{array}
\usepackage{amsthm}
\usepackage{subcaption}
\usepackage[noend]{algpseudocode}
\usepackage{graphicx}
\usepackage{color}
\usepackage{listings}
\usepackage{tikz}
\usepackage{hyperref}
\usepackage[symbol]{footmisc}

\newcommand{\N}{\ensuremath{\mathbb{N}}}

\newcommand{\R}{\ensuremath{\mathbb{R}}}

\newcommand{\tT}{\mathrm{T}}
\newcommand{\HH}{\mathcal{H}}
\newcommand{\HT}{{\mathcal{H}_{T}}}
\newcommand{\KK}{\mathcal{K}}

\newmuskip\pFqmuskip

\newcommand*\pFq[6][8]{
  \begingroup 
  \pFqmuskip=#1mu\relax
  \begingroup\lccode`\~=`\,
  \lowercase{\endgroup\let~}\pFqcomma
  {}_{#2}F_{#3}{\left(\genfrac..{0pt}{}{#4}{#5};#6\right)}%
  \endgroup
}
\newcommand*\pRegFq[6][8]{
  \begingroup 
  \pFqmuskip=#1mu\relax
  \begingroup\lccode`\~=`\,
  \lowercase{\endgroup\let~}\pFqcomma
  {}_{#2}\tilde{F}_{#3}{\left(\genfrac..{0pt}{}{#4}{#5};#6\right)}%
  \endgroup
}

\newcommand{\pFqcomma}{\mskip\pFqmuskip}

\DeclareMathOperator*{\argmin}{argmin}
\DeclareMathOperator*{\argmax}{argmax}
\def\prox{\mathrm{prox}}
\def\Prox{\mathrm{Prox}}
\def\St{\mathrm{St}}

\newtheorem{thm}{Theorem}[section]
\newtheorem{lemma}[thm]{Lemma}
\newtheorem{remark}[thm]{Remark}

\newtheorem{example}[thm]{Example}
\newtheorem{corollary}[thm]{Corollary}
\newtheorem{proposition}[thm]{Proposition}

\begin{document}
\title{Parseval Proximal Neural Networks}

\author{
Marzieh Hasannasab\footnotemark[1]
\and
Johannes Hertrich\footnotemark[2]
\and
Sebastian Neumayer\footnotemark[2]
	\and
	Gerlind Plonka\footnotemark[3]
	\and
	Simon Setzer\footnotemark[4]
	\and
	Gabriele Steidl\footnotemark[1]
}

\maketitle
\footnotetext[1]{Institute of Mathematics,
	TU Berlin,
	Stra{\ss}e des 17.  Juni 136, 
	D-10623 Berlin, Germany,
	\{name\}@math.tu-berlin.de.}

\footnotetext[2]{Department of Mathematics,
	Technische Universit\"at Kaiserslautern,
	Paul-Ehrlich-Str.~31, D-67663 Kaiserslautern, Germany,
	\{name\}@mathematik.uni-kl.de.}
	
	\footnotetext[3]{Institute for Numerical and Applied Mathematics,
	G\"ottingen University, Lotzestr.\ 16-18, 37083 G\"ottingen, Germany,
	plonka@math.uni-goettingen.de.
	}
	
	\footnotetext[4]{simon.setzer@gmail.com.}

\begin{abstract}
The aim of this paper is twofold. First, we show that a certain concatenation of a proximity operator with an affine operator
is again a proximity operator on a suitable Hilbert space. Second, we use our findings to establish so-called
proximal neural networks (PNNs) and stable tight frame proximal neural networks.

Let $\HH$ and $\KK$ be real Hilbert spaces, $b \in \KK$ and 
$T \in \mathcal{B} (\HH,\KK)$ a linear operator with closed range and Moore-Penrose inverse $T^\dagger$.
Based on the well-known characterization of proximity operators by Moreau, 
we prove that for any proximity operator $\Prox \colon \KK \to \KK$ the operator 
$T^\dagger \, \Prox ( T \cdot + b)$ is a proximity operator on $\HH$ equipped with a suitable norm.
In particular, it follows for the frequently applied soft shrinkage operator 
$\Prox = S_{\lambda}\colon \ell_2 \rightarrow \ell_2$ and any frame analysis operator $T\colon \HH \to \ell_2$ that the frame shrinkage operator $T^\dagger\, S_\lambda\, T$ is a proximity operator on a suitable Hilbert space.

The concatenation of proximity operators on $\mathbb R^d$ equipped with different norms establishes 
a  PNN.
If the network arises from tight frame analysis or synthesis operators, then it forms an averaged operator.
In particular, it has Lipschitz constant 1 and belongs to the class of
so-called Lip\-schitz networks, which were
recently applied to defend against adversarial attacks.
Moreover, due to its aver\-aging property, PNNs  can be used within so-called Plug-and-Play algorithms
with convergence guarantee.
In case of Parseval frames, we call the networks Parseval proximal neural networks (PPNNs). 
Then, the involved linear operators
are in a Stiefel manifold and corresponding minimization methods can be applied for training of such networks.
Finally, some proof-of-the concept examples demonstrate the performance of PPNNs.
\end{abstract}

\section{Introduction}

Wavelet and frame shrinkage operators became very popular in recent years.
A certain starting point was the iterative shrinkage-thresholding algorithm (ISTA) in \cite{DDM04}, which was interpreted as a special case of
the forward-backward algorithm  in \cite{CW05}. 
For relations with other algorithms see also~\cite{BSS2017,Se09}.
Let  $T \in  \mathbb R^{n \times d}$, $n\ge d$, have full column rank.
Then, the problem
\begin{equation} \label{eq:problem}
\argmin_{y \in \mathbb R^d} \bigl\{ \tfrac12 \|x -y \|_2^2 + \lambda \|Ty\|_{1} \bigr\}, \quad \lambda >0,
\end{equation}
is known as the analysis point of view.
For orthogonal $T \in {\mathbb R}^{d \times d}$, the solution of \eqref{eq:problem} is given by the frame soft shrinkage operator 
$T^\dagger \, S_\lambda \, T  = T^* \, S_\lambda \, T$, see Example \ref{ex:1}.
If $T \in \mathbb R^{n \times d}$ with $n \le d$ and $T T^* = I_n$, 
the solution of problem \eqref{eq:problem} is given by
$I_d -  T^* T + T^* S_\lambda T$,
see \cite[Theorem 6.15]{Beck17}.
For arbitrary $T \in {\mathbb R}^{n \times d}$, $n \ge d$, 
there are no analytic expressions for the solution of \eqref{eq:problem} in the literature.

The question whether the frame shrinkage operator  can itself  be seen as a proximity operator has been recently studied in \cite{GP2019}.
They showed that the set-valued operator $(T^\dagger S_\lambda T)^{-1} - I_d$
is maximally cyclically monotone, which implies that it is a proximity operator with respect to some norm in
$\mathbb R^d$.
In this paper, we prove that for any operator $T\in \mathcal{B} (\HH,\KK)$ with closed range, $b \in \KK$ and
any proximity operator $\Prox\colon \KK \to \KK$ 
the new operator $T^\dagger \, \Prox \, ( T \cdot+ b) \colon \HH \to \HH$ 
is also a proximity operator on the linear space $\HH$, but equipped with another inner product.
The above mentioned finite dimensional setting is included as a special case.
In contrast to \cite{GP2019}, we directly approach the problem using a classical result of Moreau \cite{Moreau65}. 
Moreover, we provide the function for the definition of the proximity operator.
Here, we like to mention that this function can be also deduced from Proposition 3.9 in \cite{Combettes2018}.
However, since this deduction appears to be more space consuming than the direct proof of our Theorem \ref{thm:ExpForm},
we prefer to give a direct approach.
Note that different norms in the definition of the proximity  operator 
were successfully used in variable metric algorithms, see \cite{CPR2013}.

Recently, it was shown that many activation functions appearing in neural networks are indeed proximity functions \cite{CP2018}.
Based on this observations and our previous findings, 
we consider neural networks that are concatenations of proximity operators and call them proximal neural networks (PNNs).
PNNs can be considered within the framework of the variational networks proposed in \cite{KKHP2017}.
Due to stability reasons, PNNs related to linear operators from Stiefel manifolds are of special interest. 
They form so-called averaged operators and are consequently nonexpansive.
Orthogonal matrices have already shown advantages in training recurrent neural networks (RNNs)
\cite{ASB2016,bansal2018can, jing2017tunable, lezcano2019cheap, Vorontsov2017,Wisdom2016}. 
Using orthogonal matrices, vanishing or exploding gradients in training RNNs can be avoided~\cite{Dorobantu2016}.
The more general setting of learning rectangular matrices from a Stiefel manifold was proposed, e.g., in \cite{HF2016}, 
but with a different focus than in this paper.
The most relevant paper with respect to our setting is \cite{huang2018orthogonal}, where the authors considered the so-called
optimization over multiple dependent Stiefel manifolds (OMDSM).
We will see that the NNs in \cite{huang2018orthogonal} are special cases of our PNNs so that our analysis ensures that they are
averaged operators.
 
Using matrices from Stiefel manifolds results in $1$-Lipschitz neural networks.
Consequently, our approach is naturally related to other methods for controlling the Lipschitz constant of neural networks, which provably increases robustness against adversarial attacks \cite{TSS2018}.
In \cite{GFPC18}, the constant is controlled by projecting back all weight matrices 
in the network that violate a pre-defined threshold on the $\Vert \cdot\Vert_p$ norm, $p \in [1,\infty]$, of the weight matrices.
The authors in \cite{SGL2019} characterize the singular values of the linear map associated
with convolutional layers and use this 
for projecting a convolutional
layer onto an operator-norm ball.
Another closely related approach is spectral normalization as proposed in \cite{MKKY2018}, 
where the spectral norm of every weight matrix is enforced to be one.
Compared to our approach, this only restricts the largest singular value of the linear operators arising in the neural network.
Limitations of the expressiveness of networks with restricted Lipschitz constants in every layer were discussed in \cite{ALG19,HCC2019}.
Note that our approach does not restrict the Lipschitz constants in every individual layer.
Further, none of the above approaches is able to impose more structure on the network such as being an averaged operator.

Our results may be of interest in so-called Plug-and-Play algorithms \cite{CWE2016,SVW2016,TBF2018}.
In these algorithms a well-behaved operator, e.g., a proximity operator, is replaced by an efficient denoiser such as a neural network.
However, training a denoising framework without structure can lead to a divergent algorithm,
see \cite{SKM2019}.
In contrast, it was shown in \cite{Sun2018AnOP} 
that a particular version of a Plug-and-Play algorithm converges if the network is averaged.

Our paper is organized as follows: 
We begin with preliminaries on convex analysis in Hilbert spaces in Section~\ref{sec:prelim}.
In Section~\ref{sec:interplay}, we prove our general results on the interplay between proximity and certain affine operators.
As a special case we emphasize that the frame soft shrinkage operator is itself a proximity operator in Section~\ref{sec:frame}.
In Section~\ref{sec:PNN}, we use our findings to set up neural networks as a concatenation of proximity operators on $\mathbb R^d$
equipped with different norms related to linear operators.
If these operators are related to tight frames, our proposed network is actually an averaged operator.
In case of Parseval frames,
the involved matrices are in Stiefel manifolds and we end up with PPNNs.
Section~\ref{sec:PPNN} deals with the training of PPNNs via stochastic gradient descent on Stiefel manifolds.
In Section~\ref{sec:numerics}, we provide first numerical examples.
Finally, Section~\ref{sec:conclusions} contains conclusions and addresses further research questions.

\section{Preliminaries} \label{sec:prelim}
Let $\HH$ be a real Hilbert space with inner product $\langle \cdot,\cdot \rangle$ and norm $\| \cdot \|$. 
By  $\Gamma_0(\HH)$ we denote the set of proper, convex, lower semi-continuous functions on $\HH$ mapping into $(-\infty,  \infty]$. 
For $f \in \Gamma_0(\HH)$ and $\lambda > 0$, the \emph{proximity operator} $\prox_{\lambda f}\colon \HH \rightarrow \HH$ 
and its \emph{Moreau envelope}
$M_{\lambda f}\colon \HH \rightarrow \R$ 
are defined by
\begin{align}\label{prox_usual}
\prox_{\lambda f} (x) &\coloneqq \argmin_{y \in \HH} \bigl\{ \tfrac12 \|x-y\|^2 + \lambda f(y) \bigr\}, \\
M_{\lambda f} (x) &\coloneqq \min_{y \in \HH} \bigl\{ \tfrac12 \|x-y\|^2 + \lambda f(y) \bigr\}.
\end{align}
Clearly, the proximity operator and its Moreau envelope depend on the underlying space~$\HH$, in particular on the chosen inner product.
Recall that an operator $A\colon \HH \to \HH$ is called \emph{firmly nonexpansive} if for all $x,y \in \HH$ the following relation
is fulfilled
\begin{equation} \label{star_ap}
\|Ax -Ay\|^2 \le \langle x-y,Ax-Ay \rangle .
\end{equation}
Obviously, firmly nonexpansive operators are nonexpansive.

For a Fr\'echet differentiable function $\Phi\colon \HH \to \R$, 
the gradient $\nabla \Phi(x)$ at $x \in \HH$ is defined as the vector satisfying for all $h \in \HH$,
\[\langle \nabla \Phi(x), h \rangle = D\Phi(x) h,\]
where $D\Phi\colon \HH\to \mathcal{B} (\HH,\R)$ denotes the Fr\'echet derivative of $\Phi$,
i.e., for all $x,h \in \HH$,
\begin{equation} \label{frechet}
\Phi(x+h) - \Phi(x) = D\Phi(x) h + o(\|h\|).
\end{equation}
Note that the gradient crucially depends on the chosen inner product in $\HH$.
The following results can be found, e.g., in \cite[Props.~12.27, 12.29]{BC11}. 

\begin{thm} \label{lem:1}
Let $f \in \Gamma_0(\HH)$. Then, the following relations hold true:
\\
i) The operator $\prox_{\lambda f} \colon \HH \to \HH$ is firmly  nonexpansive.
\\
ii) The function $M_{\lambda f}$ is  (Fr\'echet) differentiable with Lipschitz-continuous gradient given by
\[\nabla  M_{\lambda f}(x) = x - \prox_{\lambda f}(x).\]
\end{thm}

Clearly, ii) implies that
\begin{equation}\label{proxi}
\prox_{\lambda f} (x) = \nabla \bigl( \tfrac12 \| x \|^2 - M_{\lambda f}(x)  \bigr) = \nabla \Phi(x),
\end{equation}
where $\Phi \coloneqq \frac12 \|\cdot \|^2 - M_{\lambda f}$ is convex as $\prox_{\lambda f}$ is nonexpansive \cite[Prop.~17.10]{BC11}.
Further, it was shown by Moreau that also the following (reverse) statement holds true \cite[Cor.~10c]{Moreau65}.

\begin{thm} \label{thm:1}
 The operator $\Prox\colon  \HH \rightarrow \HH$ is a proximity operator 
 if and only if it is nonexpansive and there exists a function $\Psi \in \Gamma_0(\HH)$ with $\Prox(x) \in \partial \Psi(x)$ for any $x \in \HH$, where $\partial \Psi$ denotes the subdifferential of $\Psi$.
\end{thm}

Thanks to \eqref{proxi}, we conclude that $\Prox\colon  \HH \rightarrow \HH$ is a proximity operator 
if and only if it is nonexpansive and the gradient of a convex, differentiable function $\Phi \colon \HH \to \R$.
Recently, the characterization of Bregman proximity operators in a more general setting was discussed in \cite{GN2018}.
In the following example, we recall the Moreau envelope and the proximity operator related to the soft thresholding operator.

\begin{example} \label{ex:1}
	Let $\HH = \R$ with usual norm $|\cdot|$ and $f(x) \coloneqq |x|$.
	Then, $\prox_{\lambda f}$ is the soft shrinkage operator $S_\lambda$ defined by
	\[S_\lambda(x)\coloneqq 	\left\{
	\begin{array}{cl}
	x - \lambda& \mathrm{for} \; x > \lambda,\\
	0         & \mathrm{for} \; x \in [-\lambda,\lambda],\\
	x + \lambda& \mathrm{for} \; x < -\lambda,
	\end{array}
	\right.\]
	and the Moreau envelope is the Huber function
	\[
	m_{\lambda | \cdot|} (x)
	=
	\left\{
	\begin{array}{cl}
	\lambda x - \frac{\lambda^2}{2} & \mathrm{for} \; x > \lambda,\\[0.5ex]
	\frac{1}{2} x^2         & \mathrm{for} \; x \in [-\lambda,\lambda],\\[0.5ex]
	- \lambda x - \frac{\lambda^2}{2}& \mathrm{for} \; x < -\lambda.
	\end{array}
	\right.\]
	Hence, $\prox_{\lambda f} = \nabla \varphi$, where $\varphi(x)  = \frac{x^{2}}{2} - m_{\lambda | \cdot|}(x)$, i.e., 
	\[\varphi(x) = 
	\left\{\begin{array}{cl}
	\tfrac12(x-\lambda)^2& \mathrm{for} \; x > \lambda,\\
	0 &\mathrm{for} \; x \in [-\lambda,\lambda],\\[0.5ex]
	\tfrac12 (x + \lambda)^2& \mathrm{for} \; x < -\lambda.
	\end{array}
	\right.\]
	For $\HH = \mathbb R^d$ and $f(x) \coloneqq \|x\|_1$, we can use a componentwise approach.
	Then, $S_\lambda$ is defined componentwise, the Moreau envelope reads as
	$M_{\lambda \| \cdot \|_1} (x) = \sum_{i=1}^d m_{\lambda | \cdot|}(x_i)$ 
	and the potential of $\prox_{\lambda \| \cdot \|_1}$ is $\Phi(x) = \sum_{i=1}^d \varphi(x_i)$.
\end{example}

\section{Interplay Between Proximity and Linear Operators} \label{sec:interplay}
Let $\HH$ and $\KK$ be real Hilbert spaces with inner products $\langle \cdot,\cdot \rangle_\HH$ and  $\langle \cdot,\cdot \rangle_\KK$ and corresponding norms
$\| \cdot \|_\HH$ and $\| \cdot \|_\KK$, respectively.
By $\mathcal{B}(\HH,\KK)$ we denote the space of bounded, linear operators from $\HH$ to $\KK$.
The kernel and the range of 
$T\in \mathcal{B}(\HH,\KK)$ are denoted by $\mathcal N(T)$ and $\mathcal R(T)$, respectively.
In this section, we show that for any nontrivial operator $T \in \mathcal{B}(\HH,\KK)$ with closed range $\mathcal{R}(T)$, $b \in \KK$ 
and proximity operator $\Prox \colon \KK \to \KK$, 
the operator $T^\dagger \, \Prox ( T \cdot + b)\colon \HH \rightarrow \HH$ 
is itself a proximity operator on the linear space $\HH$ equipped with a suitable (equivalent) norm $\| \cdot \|_{\HT}$, 
i.e., there exits a function $f \in \Gamma_0(\HH)$ such that
\[
T^\dagger \, \Prox ( Tx + b) = \argmin_{y \in \HH} \bigl\{ \tfrac12 \|x-y\|_\HT^2 + f(y) \bigr\}.
\]

Throughout this section, 
let $T \in \mathcal{B}(\HH,\KK)$ have closed range.
Then, the same holds true for its adjoint $T^*\colon \KK \to \HH$ and the following (orthogonal) decompositions hold
\begin{equation}\label{orth}
 \KK = \mathcal{R}(T) \oplus \mathcal{N}(T^*), \qquad \HH = \mathcal{R}(T^*) \oplus \mathcal{N}(T).
\end{equation}
The Moore-Penrose inverse (generalized inverse, pseudo-inverse) $T^\dagger \in \mathcal{B}(\KK,\HH)$
is given point-wise by
$$\{T^{\dagger}y\} = \{x \in \HH: T^*\,T x = T^*y\} \cap \mathcal{R}(T^*),$$
see \cite{BC11}.
Further, it satisfies $\mathcal{R} (T^\dagger) = \mathcal{R} (T^*)$ and
\begin{equation} \label{MP1}
T^{\dagger}\, T = P_{\mathcal{R}(T^*)}, \quad T\, T^{\dagger} = P_{\mathcal{R}(T)}, 
\end{equation}
where $P_{C}$ is the orthogonal projection onto the closed, convex set $C$, see \cite[Prop.~3.28]{BC11}.
Then, it follows 
\begin{equation} \label{MP3}
T^\dagger\, T\, T^* = P_{\mathcal{R}(T^*)}\,T^* = T^* 
\quad \mathrm{and} \quad
T^\dagger\, P_{\mathcal{R}(T)} = 
T^\dagger\, T\, T^\dagger = T^\dagger.
\end{equation}
If $T$ is injective, then $T^\dagger = (T^*\,T)^{-1} T^*$
and if $T$ is surjective, we have $T^\dagger =  T^*(T\,T^*)^{-1}$.

Every $T\in \mathcal{B}(\HH,\KK)$ gives rise to an inner product in $\HH$ via
\begin{equation} \label{inner_neu}
\langle x, y \rangle_\HT = \langle Tx, Ty \rangle_\KK/\Vert T\Vert^2_{\mathcal{B}(\HH,\KK)} + \langle x, P_{\mathcal{N}(T)} y \rangle_\HH
\end{equation}
with corresponding norm
\[ \|x\|_\HT= \bigl(\| Tx\|^2_\KK/\Vert T\Vert^2_{\mathcal{B}(\HH,\KK)} + \|  P_{\mathcal{N}(T)} x\|^2_\HH \bigr)^{\frac{1}{2}}.\]
If $T$ is injective, the second summand vanishes.
In general, this norm only induces a pre-Hilbert structure.
Since $T \in \mathcal{B}(\HH,\KK)$ has closed range, the norms $\| \cdot\|_\HH$ and $\| \cdot\|_\HT$ are equivalent on $\HH$ due to
\begin{equation} \label{new_norm}
\| x\|_\HT^2 = \Vert T x \Vert_\KK^2/\Vert T\Vert^2_{\mathcal{B}(\HH,\KK)} +  \|  P_{\mathcal{N}(T)} x\|_\HH^2 \leq 2 \Vert x \Vert_\HH^2
\end{equation}
and
\[
\Vert x \Vert_\HH^2 = \Vert T^\dagger\,T x \Vert_\HH^2 
+\Vert  P_{\mathcal{N}(T)} x \Vert_\HH^2 \leq \bigl(\Vert T^\dagger \Vert^2_{\mathcal{B}(\KK,\HH)}\Vert T\Vert^2_{\mathcal{B}(\HH,\KK)} + 1\bigr) \Vert x \Vert_\HT^2
\]
for all $x \in \HH$. The norm equivalence also ensures the completeness of $\HH$ equipped with the new norm.
To emphasize that we consider the linear space $\HH$ with this norm, we write $\HT$.
For special $T \in \mathcal{B}(\HH,\KK)$, the inner product \eqref{inner_neu} coincides with the one in $\HH$.

\begin{lemma}\label{lem:tight}
Let $T \in \mathcal{B}(\HH,\KK)$ 
fulfill 
$T^* T =  \Vert T\Vert^2_{\mathcal{B}(\HH,\KK)} \mathrm{Id}_\HH$ or $T T^* =  \Vert T\Vert^2_{\mathcal{B}(\HH,\KK)} \mathrm{Id}_\KK$, 
where $\mathrm{Id}_\HH$ and $\mathrm{Id}_\KK$ denote the identity operator on $\HH$ and $\KK$, respectively.
Then, the inner product \eqref{inner_neu} coincides with the one in $\HH$ and consequently $\HH = \HH_T$.
\end{lemma}

\begin{proof}
If $T^* T =  \Vert T\Vert^2_{\mathcal{B}(\HH,\KK)}\mathrm{Id}_\HH$, then $T$ is injective such that \eqref{inner_neu} implies
$$
\langle x, y \rangle_\HT =  \langle Tx, Ty \rangle_\KK/\Vert T\Vert^2_{\mathcal{B}(\HH,\KK)} = \langle x, T^* Ty \rangle_\HH/\Vert T\Vert^2_{\mathcal{B}(\HH,\KK)} = \langle x, y \rangle_\HH.
$$
If $T T^* =  \Vert T\Vert^2_{\mathcal{B}(\HH,\KK)}\mathrm{Id}_\KK$, 
then \eqref{MP3} implies that
$P_{\mathcal{N}(T)} = \mathrm{Id}_\HH - T^\dagger T = \mathrm{Id}_\HH - T^* T/\Vert T\Vert^2_{\mathcal{B}(\HH,\KK)}$ 
and
\begin{align*}
\langle x, y \rangle_\HT 
&=  \langle Tx, Ty \rangle_\KK/\Vert T\Vert^2_{\mathcal{B}(\HH,\KK)} +  \langle  x, P_{\mathcal{N}(T)} y \rangle_\HH\\
&= \langle x,T^*Ty \rangle_\HH/\Vert T\Vert^2_{\mathcal{B}(\HH,\KK)} + \langle x,y\rangle_\HH -  \langle x,T^* T y\rangle_\HH/\Vert T\Vert^2_{\mathcal{B}(\HH,\KK)}\\
&= \langle x,y\rangle_\HH.
\end{align*}
\end{proof}

To apply the characterization of proximal mappings in $\HH_T$ by Moreau, see Theorem~\ref{thm:1}, we have to compute gradients in $\HH_T$.
Here, the following result is crucial.

\begin{lemma}\label{lem:grad} Let $\HH$ and $\KK$ be real Hilbert spaces with inner products $\langle \cdot,\cdot \rangle_\HH$ 
and $\langle \cdot,\cdot \rangle_\KK$, respectively.
For an operator $T \in \mathcal{B}(\HH,\KK)$ with closed range, let $\HH_T$ be the Hilbert space with inner product~\eqref{inner_neu}.
For (Fr\'echet) differentiable $\Phi\colon\HH \rightarrow \mathbb R$, the gradients $\nabla_\HH \Phi$ and $\nabla_{\HH_T} \Phi$ with respect to the different inner products are related by
\[
\bigl(T^* T/\Vert T\Vert^2_{\mathcal{B}(\HH,\KK)} +  P_{\mathcal{N}(T)}\bigr)\nabla_\HT \, \Phi(x) = \nabla_\HH \Phi(x).
\]
\end{lemma}

\begin{proof}
The gradient $\nabla_\HT \Phi(x)$ 
at $x \in \HH$ in the space $\HT$ is given by the vector satisfying
\[\langle \nabla_\HT \Phi(x), h \rangle_\HT = D\Phi(x) h = \langle \nabla_\HH \Phi(x), h \rangle_\HH\]
for all $h \in \HH$.
Since 
\begin{align}
\langle \nabla_\HT \Phi(x), h \rangle_\HT &= \langle T \nabla_\HT \Phi(x), Th \rangle_\KK /\Vert T\Vert^2_{\mathcal{B}(\HH,\KK)} 
+ \langle \nabla_\HT \Phi(x),P_{\mathcal{N}(T)} h \rangle_\HH\\
&= \langle T^* T \nabla_\HT \Phi(x), h \rangle_\HH/\Vert T\Vert^2_{\mathcal{B}(\HH,\KK)} + \langle P_{\mathcal{N}(T)} \nabla_\HT \Phi(x), h \rangle_\HH,
\end{align}
the gradient depends on the chosen inner product through
$$
\bigl(T^* T/\Vert T\Vert^2_{\mathcal{B}(\HH,\KK)} +  P_{\mathcal{N}(T)}\bigr)\nabla_\HT \, \Phi(x) = \nabla_\HH \Phi(x).
$$
\end{proof}

Now, the desired result follows from the next theorem.

\begin{thm}\label{thm:Existence}
Let $b \in \KK$, $T \in \mathcal{B}(\HH,\KK)$ have closed range and $\Prox \colon \KK \to \KK$
be a proximity operator on $\KK$.
Then, the operator $A \coloneqq T^\dagger \, \Prox \, (T \cdot + b) \colon \HT \to \HT$ is a  proximity operator.
\end{thm}

\begin{proof}
In view of Theorems \ref{lem:1} and \ref{thm:1}, 
it suffices to show that $A$
is nonexpansive and that there exists a convex function $\Psi \colon \HT \to \R$ with  $A  = \nabla_\HT \Psi$.

1.~First, we show that $A$ is firmly nonexpansive, and thus nonexpansive.
By \eqref{orth}, we see that
\begin{equation} \label{MP2}
P_{\mathcal{N}(T)} T^\dagger = 0.
\end{equation}
Using this and \eqref{MP1}, it follows
\begin{align}
\Vert Ax - Ay \Vert_\HT^2 
&= \frac{\Vert  T T^\dagger \left(  \, \Prox \,  (T x + b) -  \Prox \,  (T y + b) \right)\Vert_\KK^2}{\Vert T\Vert^2_{\mathcal{B}(\HH,\KK)}} + \|P_{\mathcal{N}(T)} (Ax -Ay)\|^2_\HH \\
&\leq \frac{\Vert \Prox \, (T x + b) -   \Prox \, (T y + b) \Vert_\KK^2}{\Vert T\Vert^2_{\mathcal{B}(\HH,\KK)}}. \label{xx}
\end{align}
By \eqref{MP2} and \eqref{MP1}, we obtain
\begin{align*}
\Vert T\Vert^2_{\mathcal{B}(\HH,\KK)}\bigl \langle  A x - A y, x-y \bigr \rangle_\HT
&= 
\bigl\langle T T^\dagger \bigl( \Prox \, (T x +b) - \Prox \, (T y + b) \bigr), Tx  - Ty  \bigr\rangle_\KK\\ 
&=  
\bigl\langle  P_{\mathcal{R}(T)} \bigl( \Prox \, (T x + b) - \Prox \, (T y + b) \bigr), Tx  - Ty \bigr\rangle_\KK\\
&=
\bigl\langle \Prox \, (T x + b) - \Prox \, (T y + b), Tx - Ty \bigr\rangle_\KK\\
&=
\bigl\langle \Prox \, (T x + b) - \Prox \, (T y + b), Tx + b - (Ty + b) \bigr\rangle_\KK,
\end{align*}
and since $\Prox$ is  firmly nonexpansive with respect to $\|\cdot\|_{\KK}$, see \eqref{star_ap},
the estimate \eqref{xx} further implies that $A$ is firmly nonexpansive
\begin{align}
\bigl \langle Ax - Ay , x-y  \bigr \rangle_\HT
&\geq 
\frac{\Vert \Prox \, (T x + b) - \Prox \, (T y + b) \Vert^2_\KK}{\Vert T\Vert^2_{\mathcal{B}(\HH,\KK)}}
\geq 
\Vert Ax - Ay \Vert_\HT^2.
\end{align}

2.~It remains to prove that there exists a convex function $\Psi\colon \HT \to \R$ with
$\nabla_\HT \Psi = A$.
Since $\Prox$ is a proximity operator, there exists $\Phi \colon \HH \to\R$ with $\Prox = \nabla_\KK \Phi$.
Then, a natural candidate is given by $\Psi =\Phi \, (T \cdot + b)/\Vert T\Vert^2_{\mathcal{B}(\HH,\KK)}$.
Using the definition of the gradient and the chain rule, it holds for all $x,h\in \HH$ that
\begin{align}
\langle \nabla_\HH \Psi(x) , h \rangle_\HH &=  
D\Psi (x)h = \frac{D\Phi(Tx + b)\,Th}{\Vert T\Vert^2_{\mathcal{B}(\HH,\KK)}}= 
\frac{\langle \nabla_\KK \Phi(Tx + b) , Th \rangle_\KK}{\Vert T\Vert^2_{\mathcal{B}(\HH,\KK)}}\\ &= \frac{\langle T^* \Prox \, (Tx + b)  , h \rangle_\HH}{\Vert T\Vert^2_{\mathcal{B}(\HH,\KK)}}.
\end{align}
Incorporating Lemma~\ref{lem:grad}, we conclude
$$
(T^* T/\Vert T\Vert^2_{\mathcal{B}(\HH,\KK)} +  P_{\mathcal{N}(T)}) \nabla_\HT \Psi  = \nabla_\HH \Psi(x) = T^* \, \Prox \, (Tx+ b)/\Vert T\Vert^2_{\mathcal{B}(\HH,\KK)},
$$ 
which implies 
$T^*\, T\, \nabla_\HT \Psi = T^* \, \Prox \, (Tx + b)$ and $\nabla_\HT \Psi \in \mathcal{R}(T^*)$.
By definition of $T^\dagger$, we obtain $\nabla_\HT \Psi = A$.
Finally, $\Psi$ is convex as it is the concatenation of a convex function with a linear function.
\end{proof}

Let 
$$(f\square g)(x) \coloneqq \inf_{y\in \HH } f(y) + g(x-y)$$ 
denote the \emph{infimal convolution} of $f,g \in \Gamma_0(\HH)$ and 
$x \mapsto \iota_S(x)$ the \emph{indicator function} of the set $S$ taking the value $0$ if $x \in S$ and $+\infty$ otherwise.

For $\Prox \coloneqq \prox_{g}$ 
with $g \in \Gamma_0(\HH)$,
we are actually able to explicitly compute 
$f \in \Gamma_0(\HH)$ such that $T^\dagger \, \Prox \, (T \cdot + b) = \prox_f$ on $\HT$.
Clearly, this also gives an alternative proof for Theorem~\ref{thm:Existence}.

\begin{thm}\label{thm:ExpForm}
Let $b \in \KK$, $T \in \mathcal{B}(\HH,\KK)$ with closed range 
and $\Prox \coloneqq \prox_{g}$ for some $g \in \Gamma_0(\KK)$.
Then, $T^{\dagger} \, \prox_{g} \, (T \cdot + b) \colon \HT \to \HT$ is the proximity operator on $\HT$ of $f \in \Gamma_0(\HH)$ given by
\begin{equation} \label{prox_expl}
f(x) 
\coloneqq 
g  \square \bigl( \tfrac12 \| \cdot \|_{\KK}^2 + \iota_{\mathcal{N}(T^*)}  \bigr)  (Tx + b)/\Vert T\Vert^2_{\mathcal{B}(\HH,\KK)} + \iota_{\mathcal{R}(T^*)}(x).
\end{equation}
This expression simplifies to
\begin{align}
 f(x) &=  g  \square \bigl( \tfrac12 \| \cdot \|_{\KK}^2 + \iota_{\mathcal{N}(T^*)}  \bigr)  (Tx + b)/\Vert T\Vert^2_{\mathcal{B}(\HH,\KK)} \quad \mbox{if $T$ is injective},\\
 f(x) &=  g(Tx + b)/\Vert T\Vert^2_{\mathcal{B}(\HH,\KK)} + \iota_{\mathcal{R}(T^*)}(x) \quad \mbox{if $T$ is surjective},\\
 f(x) &=  g(Tx + b)/\Vert T\Vert^2_{\mathcal{B}(\HH,\KK)}  \quad \mbox{if $T$ is bijective}.
\end{align}
\end{thm}

\begin{proof}
By \eqref{MP3} and \eqref{orth}, we obtain
\begin{align}
&T^\dagger \,  \prox_{g} \, (T x + b)\\ 
=&
T^\dagger \argmin_{z \in \KK} \bigl\{ \tfrac12 \|z - Tx - b\|_{\KK}^2 +  g(z)\bigr\}\\
=&
T^\dagger P_{\mathcal{R}(T)} \argmin_{z_1 \in \mathcal{R}(T), z_2 \in \mathcal{N}(T^*)}
\bigl\{ \tfrac12 \|z_1 + z_2 - Tx\|_{\KK}^2 +  g(z_1 + z_2 + b) \bigr\}\\
=&
T^\dagger \argmin_{z_1 \in \mathcal{R}(T)} \inf_{z_2 \in  \mathcal{N}(T^*)}  
\bigl\{ \tfrac12 \|z_1 - Tx\|_{\KK}^2 + \tfrac12 \|z_2\|_{\KK}^2 +  g(z_1 + z_2 + b)  \bigr\}\\
=& 
T^\dagger \argmin_{z_1 \in \mathcal{R}(T)} \Bigl\{ \tfrac12 \|z_1 - Tx\|_{\KK}^2 
+ \inf_{z_2 \in  \mathcal{N}(T^*)} \bigl\{ \tfrac12\|z_2\|_{\KK}^2 + g(z_1 + z_2 + b)  \bigr\} \Bigr\}\\
=&
T^\dagger T \argmin_{y \in \mathcal{R}(T^*)} \Bigl\{ 
\tfrac12 \|Ty - Tx\|_{\KK}^2 
+ 
\inf_{z_2 \in  \mathcal{N}(T^*)} 
\bigl\{ \tfrac12\|z_2\|_{\KK}^2 + g(Ty + z_2 + b) \bigr\}
\Bigr\}
\end{align}
and by \eqref{MP1} further
\begin{align}
&T^\dagger \,  \prox_{g} \, (T x + b)\\ 
=&
\argmin_{y \in \mathcal{R}(T^*)} \Bigl\{ 
\tfrac12 \|Ty - Tx\|_{\KK}^2 
+ 
\inf_{z_2 \in  \mathcal{N}(T^*)} 
\bigl\{ \tfrac12\|z_2\|_{\KK}^2 + g(Ty + z_2 + b) \bigr\}
\Bigr\}\\
=&
\argmin_{y \in \mathcal{R}(T^*)} 
\Bigl\{ \tfrac12 \|y - x\|_\HT^2 + \inf_{z_2 \in  \mathcal{N}(T^*)} 
\bigl\{ \tfrac12\|z_2\|_{\KK}^2 
+  g(Ty + z_2 + b)  \bigr\}/\Vert T\Vert^2_{\mathcal{B}(\HH,\KK)} \Bigr\} \label{magic}\\
=& 
\argmin_{y \in \HH} \Bigl\{ \tfrac12 \|y - x\|_\HT^2 + 
g \square \bigl( \tfrac12 \| \cdot \|_{\KK}^2 + \iota_{\mathcal{N}(T^*)} \bigr) (Ty + b)/\Vert T\Vert^2_{\mathcal{B}(\HH,\KK)}  + \iota_{\mathcal{R}(T^*)}(y)\Bigr\}.
\end{align}
Hence, we conclude that $T^{\dagger} \, \prox_{g} \, (T \cdot + b)$ 
is the proximity operator on $\HT$ of $f$ in~\eqref{prox_expl}.
\end{proof}

Note that for surjective $T$ and $b=0$, the function $f$ is in general a weaker regularizer than $g$. 
This is necessary since for the latter \eqref{magic} would lead to
\begin{align*}
\argmin_{y \in \mathcal{R}(T^*)} \bigl\{\tfrac12 \| x-y\|_{\HH_T}^2 +   g(Ty)/\Vert T\Vert^2_{\mathcal{B}(\HH,\KK)} \bigr\} 
&= T^\dagger \argmin_{z \in \KK} 
\bigl\{ \tfrac12 \|z - Tx\|_\KK^2 + g(z) +  \iota_{{\cal R}(T)}(z)\bigr\}\\&\neq T^\dagger \prox_{g} (T x).
\end{align*}

\section{Frame Soft Shrinkage as Proximity Operator} \label{sec:frame}
In this section, we investigate the frame soft shrinkage as a special proximity operator.
Let $\KK = \ell_2$ be the Hilbert space of square summable sequences $c = \{c_k\}_{k \in \mathbb N}$ with norm 
$\|c \|_{\ell_2} \coloneqq ( \sum_{k \in \mathbb N} |c_k|^2)^{\frac12}$ and assume that $\HH$ is separable.
A sequence $\{x_k\}_{k\in\N}$, $x_k \in \HH$, is called a \emph{frame} of $\HH$, if constants $0 < A \le B < \infty$ exist such that
for all $x \in \HH$,
\begin{equation} \label{frame}
A \|x\|_\HH^2 \le \sum_{k\in \N} |\langle x,x_k \rangle_\HH |^2 \le B \|x\|_\HH^2.
\end{equation}
Given a frame $\{x_k\}_{k\in\N}$  of $\HH$, 
the  corresponding \emph{analysis operator}  $T \colon \HH \to\ell_2$ is defined as 
$$Tx=\bigl\{ \langle x,x_k \rangle_\HH\bigr\}_{k\in\N}, \quad x\in \HH.$$
Its adjoint $T^*\colon\ell_2 \to \HH$ is the \emph{synthesis operator} given by 
$$T^*\{c_k\}_{k\in\N} = \sum_{k\in\N}  c_k x_k, \quad \{c_k\}_{k\in\N} \in\ell_2.$$ 
By composing $T$ and $T^*$, we obtain the \emph{frame operator} 
$$T^*Tx = \sum_{k\in\N} \langle x , x_k \rangle_\HH x_k, \quad x\in \HH,$$
which is invertible on $\HH$, see \cite{CB2016}, such that
\[x= \sum_{k\in\N} \langle x , x_k \rangle_{\mathcal H} (T^*T)^{-1} x_k, \quad x\in \HH.\]
The sequence $\{ (T^*T)^{-1}x_k\}_{k\in\N}$ is called the \emph{canonical dual frame} of $\{ x_k\}_{k\in\N}$. 
If 
$$T^*T = \|T^* T\| \mathrm{Id}_\HH = \|T\|^2 \mathrm{Id}_\HH,$$
then $\{x_k\}_{k \in \mathbb N}$ is called a \emph{tight frame}, 
and for $T^*T = \mathrm{Id}_\HH$  a \emph{Parseval frame}.
Here, Lemma~\ref{lem:tight} comes into the play.
Note that $T^\dagger$ is indeed the synthesis operator for the canonical dual frame of $\{ f_k\}_{k\in\N}$.
The relation between linear, bounded, injective operators of closed range and frame analysis operators is given in the next proposition.

\begin{proposition} \label{prop:1} 
	\begin{itemize}
		\item[i)]
		An operator $T \in \mathcal{B}(\HH,\ell_2)$ is injective and has closed range if and only if it is the analysis operator of some frame of $\HH$. 
		\item[ii)] An operator $T \in \mathcal{B}(\ell_2,\HH)$ is surjective if and only if it is the synthesis operator of some frame of $\HH$. 
	\end{itemize}
\end{proposition}

\begin{proof}
	$i)$ If $T$ is the analysis operator for a frame $\{x_k\}_{k\in\N}$, 
	then  $T$ is bounded, injective and has closed range, see \cite{CB2016}.
	Conversely, assume that $T \in \mathcal{B}(\HH,\ell_2)$ is injective and that $\mathcal{R}(T)$ is closed.
	By \eqref{orth}, it holds $\mathcal{R}(T^*) = \mathcal H$.
	Let $\{\delta_k\}_{k\in\N}$ be the canonical basis of $\ell_2$ 
	and set 
	$\{x_k \}_{k\in\N}\coloneqq \{T^{*} \delta_k\}_{k\in\N}$. 
	Since $\sum_{k\in \N} |\langle x,x_k \rangle_\HH |^2 = \Vert Tx \Vert_{\ell_2}^2$, 
	we conclude that $\{x_k \}_{k\in\N}$ is a frame of $\HH$ and that $T$ is the corresponding analysis operator.\\
	$ii)$ Let $\{x_k\}_{k\in\N} = \{T\delta_k\}_{k\in\N}$.
	Then, the result follows from \cite[Thm.~5.5.1]{CB2016}.
\end{proof}

The soft shrinkage operator $S_\lambda$ on $\ell_2$ (applied componentwise)  
is the proximity operator corresponding to the function $g \coloneqq \lambda \| \cdot \|_1$, $\lambda>0$.
As immediate consequence of Theorem~\ref{thm:ExpForm} we obtain the following corollary.

\begin{corollary}\label{cor:f1}
	Assume that $T\colon \HH \rightarrow \ell_2$ is an analysis operator for some frame of $\HH$ and  $\Prox\colon \ell_2 \to \ell_2$ is an arbitrary proximity operator. 
	Then, $T^{\dagger} \, \Prox \,  T $ is itself a proximity operator on $\HH$ equipped with the norm
	$\| \cdot \|_\HT$.
	In particular, if $\Prox \coloneqq S_\lambda$ with $\lambda >0$, then
	\begin{align*}
	&T^{\dagger} \, S_\lambda \, (T x) = \argmin_{y \in \HH} \bigl\{ \|x-y\|_\HT^2 +  f(y)\bigr\},\\ 
	&f(y) \coloneqq \lambda \|\cdot\|_1 \square \bigl( \tfrac12 \| \cdot \|_{\ell_2}^2 + \iota_{\mathcal{N}(T^*)}  \bigr) (Ty)/\Vert T\Vert^2_{\mathcal{B}(\HH,\KK)}.
	\end{align*}
\end{corollary}

Finally, let us have a look at the finite dimensional setting with $\HH \coloneqq \R^d$, $\KK \coloneqq \R^n$, $n\ge d$.
Then, we have for any $T \in \R^{n,d}$ with full rank $d$ and the proximity operator $S_\lambda$ with $\lambda >0$ on $\R^n$ that
\begin{align}
&T^{\dagger} \, S_\lambda \, \left( T (x) \right) = \argmin_{y \in \R^d} \bigl\{ \tfrac12 \|x-y\|_\HT^2 + f(y) \bigr\},\\
&f(y) \coloneqq \lambda \|\cdot\|_1 \square \bigl( \tfrac12 \| \cdot \|_{2}^2 + \iota_{\mathcal{N}(T^*)}  \bigr) (Ty)/\Vert T\Vert^2_{\mathcal{B}(\HH,\KK)}.\label{rechnen}
\end{align}

\begin{example} \label{ex:2}
	We want to compute $f$ for the matrix $T\colon \R^{1} \to \R^{2}$ 
	given by $T = ( 1 , 2)^\tT $ and the soft shrinkage operator  $S_\lambda$ on $\R^2$ with $\lambda >0$.
	Note that this example was also considered in \cite{{GP2019}}.
	By \eqref{rechnen} and since  $x = (x_1, x_2)^\tT \in \mathcal{N}(T^*)$ 
	if and only if $x_1 = -2 x_2$, we obtain
	\begin{align}
	f(y) \Vert T\Vert^2_{\mathcal{B}(\HH,\KK)}
	& =   \lambda \Vert \cdot \Vert_1 \square \bigl( \tfrac12 \| \cdot \|_2^2 + \iota_{\mathcal{N}(T^*)} (\cdot) \bigr)  (Ty)\\
	&= \min_{Ty = z+x}  \left\{ \lambda \|z\|_1 +  \tfrac12 \| x \|_{2}^2 + \iota_{\mathcal{N}(T^*)}(x) \right\}\\ 
	&=  \min_{x \in \R^2}   \left\{ \lambda\|Ty-x\|_1 +  \tfrac12 \| x \|_{2}^2+ \iota_{\mathcal{N}(T^*)}(x) \right\}\\
	& = \min_{x \in \R^2} \bigl\{ \lambda \bigl\Vert 
	(y , 2y )^\tT  - ( x_1, x_2)^\tT  \bigr\Vert_1 + \tfrac{1}{2} \Vert x \Vert_2^2 + \iota_{\mathcal{N}(T^*)}(x) \bigr\}\\
	& =  \min_{x_2 \in \R} \bigl\{
	\lambda \vert y +2x_2 \vert +  \lambda \vert2y-x_2 \vert + \tfrac{5}{2} x_2^2 \bigr\}.
	\end{align}
	Consider the strictly convex function  $g_y(x_2) = \lambda \vert y +2x_2 \vert + \lambda \vert2y-x_2 \vert + \frac{5}{2} x_2^2$.
	For $\vert y \vert \leq \frac{2}{5} \lambda$, it holds
	\begin{equation}
	0 \in \partial_{x_2}  g_y \left(-\tfrac{y}{2} \right) = [-2\lambda,2\lambda] - \lambda \mathrm{sgn}(y) -\tfrac{5}{2} y.
	\end{equation}
	Hence, by Fermat's theorem, the unique minimizer of $g_y(x_2)$  is given by $-\frac{y}{2}$.
	Consequently, we have for $\vert y \vert \leq \frac{2}{5}\lambda$ that
	\begin{equation}
	f(y) = \tfrac{1\lambda}{2} \vert y \vert + \tfrac{1}{8} y^2.
	\end{equation}
	For $\vert y \vert > \frac{2\lambda}{5}$, the function $g_y$ is differentiable in $-\frac{\lambda}{5} \mathrm{sgn}(y)$ and it holds
	\begin{equation}
	\partial_{x_2}  g_y\bigl(-\tfrac{\lambda}{5} \mathrm{sgn}(y)\bigr) = 2 \lambda \mathrm{sgn}(y) - \lambda \mathrm{sgn}(y) - \lambda \mathrm{sgn}(y) = 0.
	\end{equation}
	Therefore, for $\vert y \vert > \frac{2\lambda}{5}$, the minimizer of $g_y$ is $-\frac{\lambda}{5} \mathrm{sgn}(y)$ and
	\begin{equation}
	f(y) = \tfrac{3\lambda}{5}\vert y \vert - \tfrac{\lambda^{2}}{50}.
	\end{equation}
	Choosing, e.g., $\lambda = \frac{1}{3}$ we obtain 
	\[ f (y) = \left\{ \begin{array}{ll}
	 \frac{1}{6} \vert y \vert + \frac{1}{8} y^{2} & \vert y \vert \le \frac{2}{15} \\[1ex]
	 \tfrac15\vert y \vert - \frac{1}{450} & \vert y \vert > \frac{2}{15} \end{array} \right.,
	 \]
	 which is a good approximation of $\tfrac15 \vert y \vert$. 
\end{example}

\section{Proximal Neural Networks} \label{sec:PNN}
In this section, we  consider neural networks  (NNs) consisting of  $K \in {\mathbb N}$ layers with dimensions $n_{1}, \ldots , n_{K}$ defined by mappings 
$\Phi = \Phi(\cdot\,; u)\colon {\mathbb R}^{d} \to {\mathbb R}^{n_{K}}$ of the form 
\begin{equation}\label{phi}
\Phi\left(x;u\right) \coloneqq A_K \sigma \circ A_{K-1}\sigma \circ \dots \sigma \circ A_1(x).
\end{equation}
Such NNs are composed  of  affine functions $A_{k}\colon {\mathbb R}^{n_{k-1}} \to {\mathbb R}^{n_{k}}$ given by
\begin{equation} \label{act1} 
A_{k}(x) \coloneqq L_{k} x + b_{k}, \qquad  k =1,\ldots, K, 
\end{equation}
with weight matrices $L_{k} \in {\mathbb R}^{n_{k}, n_{k-1}}$, $n_{0}=d$, bias vectors $b_{k} \in {\mathbb R}^{n_{k}}$ as well as a non-linear activation $\sigma\colon\mathbb{R} \rightarrow \mathbb{R}$ acting at each component, i.e., for ${x}= (x_{j})_{j=1}^{n}$ we have $\sigma (x) = (\sigma(x_{j}))_{j=1}^{n}$.
The parameter set $u\coloneqq\left(L_k,b_k\right)_{k=1}^{K}$ of such a NN has the overall dimension $D\coloneqq n_0 n_1 + n_1 n_2 + \dots + n_{K-1}n_K + n_1 + \dots + n_K$.
For an illustration see Fig.~\ref{fig:neuralnet}.

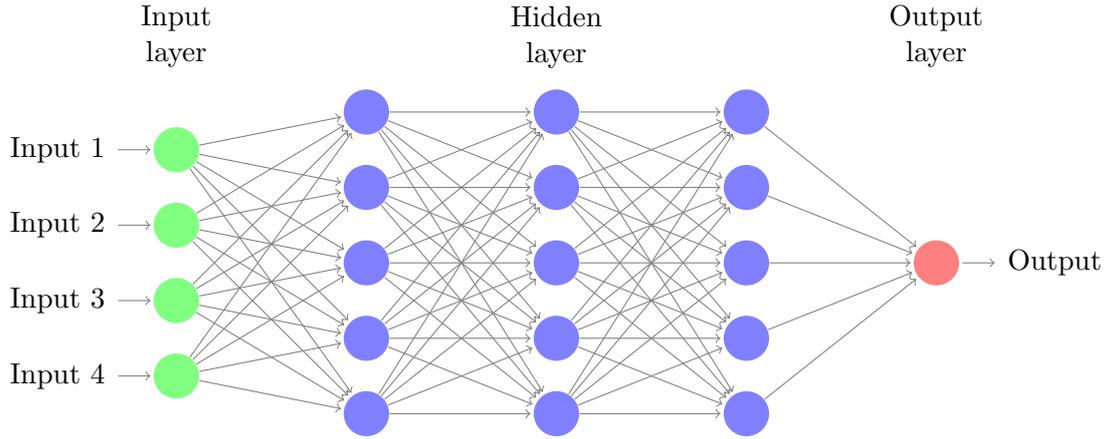
\begin{figure}[t]
\def\layersep{2.5cm}
\centering
\begin{tikzpicture}[shorten >=1pt,->,draw=black!50, node distance=\layersep]
    \tikzstyle{every pin edge}=[<-,shorten <=1pt]
    \tikzstyle{neuron}=[circle,fill=black!25,minimum size=17pt,inner sep=0pt]
    \tikzstyle{input neuron}=[neuron, fill=green!50];
    \tikzstyle{output neuron}=[neuron, fill=red!50];
    \tikzstyle{hidden neuron}=[neuron, fill=blue!50];
    \tikzstyle{annot} = [text width=4em, text centered]

    \foreach \name / \y in {1,...,4}
        \node[input neuron, pin=left:Input \y] (I-\name) at (0,-\y) {};

    \foreach \name / \y in {1,...,5}
        \path[yshift=0.5cm]
            node[hidden neuron] (Ha-\name) at (\layersep,-\y cm) {};

    \foreach \name / \y in {1,...,5}
        \path[yshift=0.5cm]
            node[hidden neuron] (Hb-\name) at (2*\layersep,-\y cm) {};
            
    \foreach \name / \y in {1,...,5}
        \path[yshift=0.5cm]
            node[hidden neuron] (Hc-\name) at (3*\layersep,-\y cm) {};        

    \node[output neuron,pin={[pin edge={->}]right:Output}, right of=Hc-3] (O) {};

    \foreach \source in {1,...,4}
        \foreach \dest in {1,...,5}
            \path (I-\source) edge (Ha-\dest);
            
    \foreach \source in {1,...,5}
        \foreach \dest in {1,...,5}
            \path (Ha-\source) edge (Hb-\dest);
            
    \foreach \source in {1,...,5}
        \foreach \dest in {1,...,5}
            \path (Hb-\source) edge (Hc-\dest);

    \foreach \source in {1,...,5}
        \path (Hc-\source) edge (O);

    \node[annot,above of=Hb-1, node distance=1cm] (hm) {Hidden layer};
    \node[annot,above of=Ha-1, node distance=1cm] (hl) {};
    \node[annot,above of=Hc-1, node distance=1cm] (hr) {};
    \node[annot,left of=hl] {Input layer};
    \node[annot,right of=hr] {Output layer};
\end{tikzpicture}
\caption{Model of a NN with three hidden layers, i.e., $d=4$, $K=4$, $n_1=n_2=n_3=5, n_4=1$.}
\label{fig:neuralnet}
\end{figure}

In \cite{CP2018}, the notation of stable activation functions was introduced.
An activation function $\sigma\colon \mathbb R \rightarrow \mathbb R$ is called \emph{stable} if it is monotone increasing, $1$-Lipschitz continuous and satisfies  $\sigma(0) = 0$.
The following result was shown in \cite{CP2018}.

\begin{lemma}\label{lem:activation}
A function $\sigma\colon\mathbb{R} \rightarrow \mathbb{R}$ is a stable activation function
if and only if there exists $g \in \Gamma_0(\mathbb R)$ having $0$ as a minimizer
such that $\sigma = \prox_{g}$.
\end{lemma}

Various common activation functions $\sigma$ and corresponding functions $g \in \Gamma_0(\mathbb R)$
are listed in Tab.~\ref{prox:act} in the appendix.
For $T_k \in \mathbb R^{n_k,d}$, we consider the norm \eqref{new_norm} and denote it by
\begin{equation} \label{norm_dis}
\|x\|_{T_{k}} \coloneqq \bigl( \|T_{k} x\|_2^2/\Vert T_{k} \Vert_2^2 + \|(I - T_{k}^\dagger T_{k})x\|_2^2 \bigr)^\frac12, \qquad x \in \mathbb R^d.
\end{equation}

In the previous sections, we have considered two different kinds of proximity operators, namely $\prox_g$ with respect to the Euclidean norm
\begin{equation} \label{prox_euclid}
\prox_{g}= \argmin_{y \in \mathbb R^d} \bigl\{ \tfrac12 \|x-y\|_{2}^2 + g(y) \bigr\},
\end{equation}
and $\prox_{T_k,g}$ with respect to the norm \eqref{norm_dis}
\begin{equation} \label{prox_T}
\prox_{T_k,g}= \argmin_{y \in \mathbb R^d} \bigl\{ \tfrac12 \|x-y\|_{T_k}^2 + g(y) \bigr\}.
\end{equation}
Further, we derived a function $f_k$ depending on $g$, $T_k$ and $b_k$, see Theorem \ref{thm:ExpForm}, such that
\begin{align} \label{multiprox}
\prox_{T_k,f_k}(x) =
\argmin_{y \in \mathbb R^d} \left\{ \tfrac12 \|x-y\|_{T_k}^2 + f_k(y) \right\} = T_k^\dagger \prox_g (T_k x + b_k).
\end{align}

Based on  our observations in  the previous sections, we consider the following special NNs.
We  choose a stable activation function $\sigma = \prox_g$ for some  $g \in \Gamma_0(\mathbb R)$
and  matrices  $T_{k} \in {\mathbb R}^{n_{k} , d}$, 
as well as bias vectors $b_{k} \in {\mathbb R}^{n_{k}}$, $k=1, \ldots , K$, and construct according to \eqref{act1} the affine  mappings
\begin{align} \label{act2}
A_{k}(x) \coloneqq  \underbrace{T_{k} T_{k-1}^{\dagger}}_{L_k} (x) + b_{k} , \qquad k=1,\ldots,K.
\end{align}
Then, the NN $\Phi\colon\mathbb{R}^d\rightarrow\mathbb{R}^{n_K}$ in \eqref{phi} with $A_{k}$, $b_{k}$  in \eqref{act2} can be rewritten as
\begin{align}
\Phi\left(x;u\right)
&= T_K \, T_{K-1}^\dagger \sigma \bigl( T_{K-1} \ldots T_2^\dagger \sigma \bigl( T_2 T_{1}^\dagger \sigma(T_1 x + b_1) + b_2 \bigr)  \ldots \bigr)+ b_K\\
&= T_K\, \prox_{T_{K-1},f_{K-1}} \circ  \ldots \circ \prox_{T_{1},f_1} (x) + b_K.\label{PNN_1}
\end{align}
We call $\Phi$ a  
\emph{proximal neural network} (PNN) with network parameters $u \coloneqq (T_k,b_k)_{k=1}^{K}$. 
\medskip

Next, we investigate stability properties of such networks.
Recall that an operator $\Psi\colon \HH \rightarrow \HH$ on a Hilbert space $\HH$ 
is \emph{$\alpha$-averaged}, $\alpha \in (0,1)$, if there exists a nonexpansive operator
$R\colon \HH \rightarrow \HH$
such that 
\[\Psi = \alpha R + (1-\alpha) I_{\HH}.\]
The following theorem summarizes properties of $\alpha$-averaged operators, c.f.~\cite{BC11} and  \cite{Reich1979}
for the third statement.

\begin{thm} \label{lem:averaged}
Let $\HH$ be a separable real Hilbert space. Then the following holds true:
\begin{itemize}
\item[i)] 
An operator on $\HH$ is firmly nonexpansive if and only if it is $\frac12$-averaged.
\item[ii)]
The concatenation of $K$ operators which are $\alpha_k$-averaged with respect to the same norm is $\alpha$-averaged 
with 
$
\alpha = \frac{K}{K-1 + 1/\max_k \alpha_k}
$.
\item[iii)]
For an $\alpha$-averaged operator $\Psi\colon \HH \rightarrow \HH$ with a nonempty fixed point set,
the sequence generated by the iteration
\[
x^{(r+1)} = \Psi\bigl(x^{(r)} \bigr)
\]
converges weakly for every starting point $x^{(0)} \in \HH$ to a fixed point of $\Psi$.
\end{itemize}
\end{thm}

In the following, we study special PNNs, which are $\alpha$-averaged operators 
such that  $x^{(r+1)} = \Phi(x^{(r)};u)$ converges to a fixed point of $\Phi$ if such a point exists.

\begin{lemma}\label{lem:special}
\begin{itemize}
\item[i)] Let
$T_k\in \mathbb R^{n_k,d}$ fulfill
$T_k^* T_k = \Vert T_k \Vert^2 I_{d}$ or $T_k T_k^* = \Vert T_k \Vert^2 I_{n_{k}}$ for all $k=1,\ldots, K-1$ and let $T_K = I_d$.
Then $\Phi$ in \eqref{PNN_1} is $\alpha$-averaged with $\alpha = \frac{K-1}{K}$.
\item[ii)]
Let $T_1 \in \mathbb R^{n_k,d}$ with full column rank fulfill $\Vert T_1\Vert^2 T_k^* T_k =\Vert T_k\Vert^2 T_1^* T_1$ for $k=1,\ldots,K-1$ and  $T_K = I_d$.
Then $\Phi$ in \eqref{PNN_1} is $\alpha$-averaged with $\alpha = \frac{K-1}{K}$.
\end{itemize}
\end{lemma}

\begin{proof}
i) By Lemma~\ref{lem:tight}, we know that $\|\cdot\|_{T_k} = \|\cdot\|_2$ so that $\Phi$ is the concatenation of $K-1$ proximity
operators on $\mathbb R^d$ with respect to the Euclidean norm. More precisely,
$$
\Phi = T_K\,\prox_{f_{K-1}} \circ  \ldots \circ \prox_{f_1} (x) + b_K
$$
with $f_k$ as in Theorem \ref{thm:ExpForm}.
Now, the assertion follows from Theorem \ref{lem:averaged}.

ii) By assumption, we obtain 
$\|x\|_{T_k} = x^* T_k^* T_k x /\Vert T_k\Vert^2 = x^* T_1^* T_1 x/\Vert T_1\Vert^2  = \|x\|_{T_1}$.
Hence, $\Phi$ becomes the concatenation of $K-1$ proximity operators on $\mathbb R^d$ all with respect to the $\Vert \cdot \Vert_{T_1}$ norm.
Again, the assertion follows from Theorem \ref{lem:averaged}.
\end{proof}

\begin{remark} 
Lemma \ref{lem:special} i) can be generalized to the case where $T_K\in\R^{d,d}$ is a symmetric positive semi-definite matrix 
with norm not larger than 1.  In this case,  $T_K$ can be written in the form $T_K=Q^*Q$ 
for some $Q\in\R^{d,d}$ with $\|Q\|^2=\|Q^*Q\| =\|T_K\|\leq 1$. Thus, for every $x,y\in \R^d$,  
\begin{align}
 \| T_Kx +b_K -(T_Ky+b_K)\|^2 &= \| Q^*Q(x-y)\|^2 \leq \| Q(x-y)\|^2 \\
 &=\langle Q(x-y), Q(x-y)\rangle\\
  &= \langle x-y, T_Kx+b_K - (T_Ky+b_K) \rangle.  
 \end{align}
This shows that $T_K\cdot+b_K$ is firmly nonexpansive and therefore $\frac12$-averaged. 
Consequently, $\Phi$ in \eqref{PNN_1} is the concatenation of $K$ $\frac12$-averaged operators
with respect to the Euclidean norm.
Hence, $\Phi$ is itself $\alpha$-averaged with $\alpha = \frac{K}{K+1}$.
\end{remark}

\begin{remark} In \cite{CP2018}, the following NN structure was studied:
Let $\HH_0, \ldots,\HH_K$ be a sequence of real Hilbert spaces and $\HH_0 = \HH_K = \HH$.
Further, let $W_k \in {\mathcal B} (\HH_{k-1},\HH_k)$ and $P_k\colon \HH_k \rightarrow \HH_k$, $k=1,\ldots,K$
be firmly nonexpansive operators.
For this case, Combettes and Pesquet \cite{CP2018} have posed conditions on $W_k$ such that
\begin{equation} \label{CP_mit}
\Psi \coloneqq W_K \circ P_{K-1} \circ W_{K-1} \circ \ldots \circ W_2 \circ P_1 \circ W_1 
\end{equation}
is $\alpha$-averaged for some $\alpha \in (1/2,1)$.
For $\HH = \mathbb R^d$ equipped with the Euclidean norm, $\HH_k = \mathbb R^d$ equipped with the norm~\eqref{norm_dis} 
and $T_K = I_d$, our PNN $\Phi$ has exactly the form~\eqref{CP_mit} with $P_k \coloneqq \prox_{T_k,f_k}\colon \HH_k  \rightarrow \HH_k$ 
and the embedding operators $W_k\colon \HH_{k-1} \hookrightarrow \HH_k$, $k=1,\ldots,K$.
For the special PNNs in Lemma \ref{lem:special} it holds $W_k = I_d$, such that the conditions in \cite{CP2018} are fulfilled.
\end{remark}
\hspace{0.5cm}

In the rest of this paper, we restrict our attention to matrices
$T_k\colon \mathbb R^{d} \rightarrow \mathbb R^{n_k}$ fulfilling
\begin{equation} \label{bed:stiefel}
T_k^* T_k = I_{d} \quad \mathrm{or} \quad T_k T_k^* = I_{n_{k} }, \quad k=1,\ldots, K-1,
\end{equation}
and $T_{K} = I_{d}$, i.e., the rows, resp.~columns of $T_k$ form a Parseval frame.
Then, the  PNN in \eqref{PNN_1} has the form 
\begin{equation} \label{PPNN}
\hspace*{-3cm} \mathrm{(PPNN)}  \qquad \qquad 
\Phi\left(x;u\right) = T_K \circ \prox_{f_{K-1}} \circ  \ldots \circ \prox_{f_1} (x) + b_K,
\end{equation} 
with the ``usual'' proximity operator, cf.~\eqref{prox_euclid}, and
\begin{align}
f_k(x) &=  g  \square \bigl( \tfrac12 \| \cdot \|_2^2 + \iota_{\mathcal{N}(T_k^*)}  \bigr)  (T_k x + b_k) \quad \mbox{if} \; d \le n_k,\\
f_k(x) &=  g(T_k x + b_k) + \iota_{\mathcal{R}(T_k^*)}(x) \quad \mbox{if} \;  d \ge n_k.
\end{align}
Due to the use of Parseval frames, we call these networks \emph{Parseval (frame) proximal neural networks} (PPNNs).
By our previous considerations, see Lemma \ref{lem:special}, PPNNs are averaged operators.

\begin{remark}\label{rem:cpp}
An interesting result follows from convergence considerations of the cyclic proximal point algorithm, see \cite{Bertsekas2011}.
Let $\{\lambda_r\}_{r \in \mathbb N} \in \ell_2 \setminus \ell_1$.
Then, for every $x^{(0)} \in \mathbb R^d$, the sequence generated by
\begin{equation} \label{cpp2}
x^{(r+1)} \coloneqq \prox_{\lambda_r f_{K-1}} \circ  \ldots \circ  \prox_{\lambda_r f_1} \bigl(x^{(r)}\bigr)
\end{equation}
converges to a minimizer of
$f_1 + \ldots + f_{K-1}$.
In particular, Theorem \ref{thm:ExpForm} implies that  for orthogonal matrices $T_k$, $k=1,\ldots,K-1$ and $T_K = I_d$, $b_K = 0$, 
the sequence $\{x^{(r)}\}_{r \in \mathbb N}$ in \eqref{cpp2} converges to 
\[
\hat x \in  \argmin_x \sum _{k=1}^{K-1} g(T_k x - b_k).
\]
\end{remark}

\section{Training PPNNs on Stiefel Manifolds} \label{sec:PPNN}
%
In this section, we show how to train PPNNs. 

\begin{remark}
According to Lemma~\ref{lem:special}~i), 
we could add more flexibility to our model 
by allowing tight frames instead.
Then, we must train an additional scaling constant, which
does not introduce difficulties in the training process 
and may be useful for special applications, see \cite{ALG19}.
In our numerical experiments, we omitted the additional scaling constant as we do not want to focus on this particular issue.
\end{remark}

In PPNNs, we assume that either $T_k$ or $T_k^*$, $k=1,\ldots,K-1$, is an element
of the Stiefel manifold
$$\St \big( \min(n_k,d ), \max(n_k,d)\big), \quad k=1,\ldots, K-1.$$
The following facts on Stiefel manifolds can be found, e.g., in \cite{AMS08}.
For $d\le n$, the (compact) Stiefel manifold is defined as
$$\St(d,n) \coloneqq \bigl\{ T \in \mathbb R ^{n,d} : T^* T = I_d\bigr\}.$$
For $d=1$ this reduces to the sphere $\mathbb S^{n-1}$, and for $d=n$ we obtain the special orthogonal group $\mathrm{SO}(n)$.
In general, $\St(d,n)$ is a manifold of dimension $nd -\frac12 d(d+1)$ with tangential space at $T \in \St(d,n)$ given by
$$
{\mathcal T}_T \mathrm{St}(d,n) =\bigl\{T U + T_\perp V: U^* = -U, V\in\R^{n-d,d}\bigr\},
$$
where the columns of $T_{\perp}\in\R^{n,n-d}$ are the basis of an orthonormal complement of $T$ fulfilling $T^*_{\perp} T_{\perp} = I_{n-d}$ and $T^*T_{\perp}=0$.
The Riemannian gradient of a function on $\St(d,n)$ can be obtained by the orthogonal projection of
the gradient in $\mathbb R^{n,d}$ onto $\St(d,n)$.
The orthogonal projection of $X \in\R^{n,d}$ onto ${\mathcal T}_T \mathrm{St}(d,n)$ is given by
\begin{align} \label{proj_stiefel_1}
P_T X 
&= (I_n - T T^*) X + \tfrac12 T(T^* X - X^* T),\\
&= WT, \qquad W \coloneqq \hat W-\hat W^*,\quad \hat W \coloneqq XT^*-\tfrac12 T(T^*XT^*). \label{proj_stiefel_2}
\end{align}
To emphasize that for fixed $T$ the matrix $W$ depends on $X$, we will also write $W_X$.
A  \emph{retraction} $\mathcal{R}$ on the manifold $\St(d,n)$ 
is a smooth mapping from the tangent bundle of $\St(d,n)$ to the manifold 
fulfilling 
${\mathcal R}_{T}(0) = T$, where $0$ is the zero element in ${\mathcal T}_T \St(d,n)$, and
with the identification ${\mathcal T}_0({\mathcal T}_{T} \St(d,n)) \cong {\mathcal T}_T \St(d,n)$ 
the local rigidity condition
$
D {\mathcal R}_{T} (0) = \mathrm{Id}_{ {\mathcal T}_T \St(d,n)}
$
holds true.
A well-known retraction on $\St(d,n)$ is
\begin{equation} \label{retraction}
\tilde {\mathcal R}_T(X)=\mathrm{qf}(T+X), \quad X \in {\mathcal T}_T\mathrm{St}(d,n),
\end{equation}
where $\mathrm{qf}(A)$ denotes the $Q$ factor of the decomposition of
a matrix $A\in\R^{n,d}$ with linearly independent columns as $A=QR$ with $Q\in \mathrm{St}(d,n)$ and $R$ 
an upper triangular matrix of size $d \times d$ with strictly positive diagonal elements.
The complexity of the QR decomposition using the Householder algorithm is $2d^2(n-d/3)$, see \cite{GL2013}.
Since the computation of the QR decomposition appears to be time consuming on a GPU, we prefer to apply another retraction, 
based on the Cayley transform  of skew-symmetric matrices $W$ in \eqref{proj_stiefel_2}, namely
\begin{equation} \label{retraction_2}
{\mathcal R}_T(X)=(I_n-\tfrac12 W)^{-1}(I_n +\tfrac12 W)T, \quad X \in {\mathcal T}_T\mathrm{St}(d,n),
\end{equation}
see \cite{NA2005,WY2013}.
By straightforward computation it can be seen that $W_X$ and $W_{P_T X}$ coincide, so that the retraction \eqref{retraction_2} enlarged to the whole $\mathbb R^{n,d}$ fulfills
\begin{equation}\label{retraction-projection}
 {\mathcal R}_T (X) = {\mathcal R}_T (P_T X), \qquad X \in \mathbb R^{n,d}.
\end{equation}

\begin{remark}\label{iteration_retraction}
The retraction \eqref{retraction_2} has the drawback that it contains a matrix inversion.
In our numerical algorithm, the following simple fixed point iteration is used for computing the matrix $R = {\mathcal R}_T(X)$ with fixed $T$ and $X$.
By definition, $R$ fulfills the fixed point equation
\begin{equation} \label{retraction_3}
R = \tfrac12 W R + (I_n + \tfrac12 W) T.
\end{equation}
Starting with an arbitrary $R^{(0)} \in \St(d,n)$, we apply the iteration
\[
R^{(r+1)} \coloneqq \tfrac12 W R^{(r)} + (I_n + \tfrac12 W) T,
\]
which converges by Banach's fixed point theorem to the fixed point of \eqref{retraction_3} if $\frac12 \rho(W) < 1$,
where $\rho(W)$ denotes the spectral radius of $W$.
\end{remark}

We want to train a PPNN by minimizing  
\begin{equation} \label{functional}
\mathcal J(u) \coloneqq  \sum_{i=1}^N {\ell} \bigl(\Phi(x_i;u); y_i\bigr),
\end{equation}
where ${\ell}\colon \mathbb R^d \times \mathbb R^d \rightarrow \mathbb R$ is a differentiable loss function on the first $d$ variables.

\begin{example}
Let us specify two special cases of PPNNs with one layer.
\begin{itemize}
 \item[i)]
For one layer without bias and componentwise soft shrinkage $\sigma$ as activation function, i.e., summands
\[
\sum_{i=1}^N {\ell} \bigl(T_1^*  \sigma(T_1 x_i);y_i \bigr), \quad T_1 \in \St(d,n_1), 
\]
we  learn Parseval frames, e.g., for denoising tasks with $y_i$ as a noisy version of $x_i$. 
Here, we want to mention the significant amount of work on dictionary learning, see \cite{EA2006}, which starts with the same goal.
\item[ii)]
For $x_i = y_i$, $i=1,\ldots,N$, the above network could be used as so-called auto-encoder.
Again, for one layer without activation function, $b_1 = 0$ and ${\ell} = h(\|x-y\|)$ 
with some norm $\| \cdot\|$ on $\mathbb R^d$ we get
\[
\sum_{i=1}^N  {\ell} \left(T_1 ^* T_1 x_i ;x_i \right) = \sum_{i=1}^N  h\bigl( \| (I_d - T_1 ^* T_1) x_i \| \bigr), \quad T_1^* \in \St(d,n_1).
\]
For the Euclidean norm and $h(x) = x^2$ we get the classical PCA approach and for $h(x) = x$
the robust rotationally invariant $L_1$-norm PCA, recently discussed in \cite{LM2018,NNSS2019}.
\end{itemize}
\end{example}

The following remark points out that special cases of our PPNNs were already considered in the literature.

\begin{remark} \label{OMDSM}
In \cite{huang2018orthogonal}, NNs with weight matrices
$L_k \in \mathbb R^{n_k,n_{k-1}}$, $k \in \{1,\ldots,K-1\}$, (or their transpose) lying in a Stiefel manifold were examined.
The authors called this approach optimization over multiple dependent Stiefel manifolds (OMDSM).
Indeed, by the following reasons, these NNs are special cases of our PPNNs if $n_k \leq d$ for all $k=1,\ldots,K-1$.
In parti\-cular, this implies that the NNs considered in \cite{huang2018orthogonal} 
(with appropriately chosen last layer) are averaged operators.
\\[1ex]
i) Case $n_{k} \le n_{k-1}$: Let  $L_k^* \in \St(n_k, n_{k-1})$, i.e.,
$L_k L_k^* = I_{n_k}$. 
Choosing an arbitrary fixed $T_{k-1} \in \mathbb R^{n_{k-1},d}$ with $T_{k-1} T_{k-1}^* = I_{n_{k-1}}$,
we want to find $T_{k} \in \mathbb R^{n_{k},d}$  such that
\begin{equation} \label{choose}
T_{k} T_{k}^* = I_{n_k} \quad \mathrm{and} \quad L_k = T_k T_{k-1}^*.
\end{equation}
It is straightforward to verify that $T_k \coloneqq L_k T_{k-1}$ has the desired properties.

Note that if the transposes of $T_k$ and $T_{k-1}$ are in a Stiefel manifold, this does not necessarily hold for the transpose of $T_k T_{k-1}^*$.
Therefore, our PPNNs are more general.
\\[1ex]
ii) Case $n_{k-1}< n_k$: Let $L_k \in \St(n_{k-1},n_k)$,
 i.e., $L_k^* L_k = I_{n_{k-1}}$. 
For an arbitrary fixed $T_{k-1} \in \mathbb R^{n_{k-1},d}$ with $T_{k-1} T_{k-1}^* = I_{n_{k-1}}$,
we want to find $T_{k} \in \mathbb R^{n_{k},d}$ fulfilling~\eqref{choose}.
To this end, we complete $L_k$ to an orthogonal matrix $\tilde L_k \in \mathbb{R}^{n_k,n_k}$ and $T_{k-1}$ to a matrix $\tilde{T}_{k-1}\in \mathbb{R}^{n_k,d}$ with orthogonal rows.
By straightforward computation we verify that $T_k \coloneqq \tilde L_k \tilde T_{k-1}$ satisfies \eqref{choose} such that this case also fits into our PPNN framework.
\end{remark}

We apply a stochastic gradient descent algorithm on the Stiefel manifold to find a minimizer of \eqref{functional}.
To this end, we compute the Euclidean gradient with respect to one layer and apply the usual backpropagation for multiple layers.

\begin{lemma} \label{grad}
Let 
\[J(T,b) \coloneqq \ell \bigl(T^* \sigma(Tx+b);y\bigr),\]
where $T$ or $T^*$ are in $\St(d,n)$, and $\ell$ and $\sigma$ are differentiable.
Set
\[r \coloneqq \sigma(Tx+b),\quad s \coloneqq T^*\sigma(Tx+b), \quad t \coloneqq  \nabla \ell \bigl(T^* \sigma(Tx+b);y\bigr), 
\quad \Sigma \coloneqq \mathrm{diag}\bigl(\sigma'(Tx + b) \bigr),\]
where the gradient of $\ell$ is taken with respect to the first $d$ variables.
Then it holds for the Euclidean gradient 
 \begin{align*}
  \nabla_T J(T,b) &= -T (t s^* + s t^*) + r t^* + \Sigma \, T t x^*,\qquad
  \nabla_b J(T,b) = \Sigma \, T t.
 \end{align*}
\end{lemma}

The proof follows by straightforward computations that are carried out in the appendix.

Now, we can formulate stochastic gradient descent (SGD) for $\mathcal J$ as in Algorithm~\ref{alg:SGD_Stiefel}.
This algorithm works for an arbitrary retraction, in particular the retraction in \eqref{retraction}.
In our numerical computations, we use the special retraction \eqref{retraction_2} in connection with the iteration scheme
\eqref{iteration_retraction}. Then, by \eqref{retraction-projection}, the projection step 3  of the algorithm can be skipped
and the retraction can be directly applied to the Euclidean gradient.

\begin{algorithm}[!ht]
\begin{algorithmic}
\State \textbf{Input:} Training data $(x_i,y_i)_{i=1}^N$, batch size $B\in\N$, learning rate $(\lambda^{(r)})_{r\in\N}$.
\State \textbf{Initialization:}  $(T^{(0)},b^{(0)})$.
\For{$r=0,1,\ldots$}
    \State 1. Choose a mini batch $I\subset\{1,\ldots,N\}$ of size $|I|=B$.
    \State 2. Compute the Euclidean gradients of 
		$\mathcal J_I(T,b) = \sum_{i\in I} \ell(\Phi(x_i;u);y_i)$ 
		\State using Lemma~\ref{grad} and backpropagation
    \begin{align}
    \nabla_b \mathcal J_I\bigl(T^{(r)},b^{(r)}\bigr) &=\sum_{i\in I}\nabla_b \ell\bigl(\Phi(x_i;u);y_i\bigr),\\
    \nabla_T \mathcal J_I\bigl(T^{(r)},b^{(r)}\bigr)&=\sum_{i\in I} \nabla_T \ell\bigl(\Phi(x_i;u);y_i\bigr).
    \end{align}
    \State 3. Compute the Riemannian gradient on the Stiefel manifold with respect to $T$ 
    \State by projection \eqref{proj_stiefel_1}.      Skip this step if retraction \eqref{retraction_2} is used
    \begin{align}
    \nabla_{\St,T} \mathcal J_I\bigl(T^{(r)},b^{(r)}\bigr)&=P_{T^{(r)}}\bigl(\nabla_T \mathcal J_I\bigl(T^{(r)},b^{(r)}\bigr)\bigr).
    \end{align}
        \State 4. Update $T$ and $b$ by a gradient descent step with retraction, see \eqref{retraction}
    \begin{align}
    \tilde T^{(r+1)} &= \tilde {\mathcal R}_{T^{(r)} } \bigl( -\tfrac{\lambda^{(r)}}{B}\nabla_{\St,T} \mathcal J_I \bigl(T^{(r)},b^{(r)}\bigr) \bigr),\\
    b^{(r+1)} &= b^{(r)}-\tfrac{\lambda^{(r)}}{B}\nabla_b \mathcal J_I\bigl(T^{(r)},b^{(r)}\bigr).
    \end{align}		
\EndFor
\end{algorithmic}

\caption{Stochastic Gradient Descent Algorithm for minimizing \eqref{functional}} \label{alg:SGD_Stiefel}
\end{algorithm}

\section{Some Numerical Results} \label{sec:numerics}
%
In this section, we present simple numerical results to get a first impression on 
the performance of PPNNs for denoising and classification.
More sophisticated examples, which include the full repertoire of fine tuning of NNs, will follow in an experimental paper, see also our conclusions.
Throughout this section, we use the quadratic loss function $\ell(x;y) \coloneqq \|x-y\|_2^2$.
For training we apply a stochastic gradient descent algorithm.
We initialize the matrices $T_k\in\St(d,n)$ randomly using the orthogonal initializer from Tensorflow. 
That is, we generate a matrix $\tilde T_k\in \R^{n,d}$ with independent random entries following the standard normal distribution and use the initialization $T_k=\mathrm{qf}(\tilde T_k)$.
The batch size and learning rate are given for all examples separately.
\medskip

\noindent
\textbf{Denoising.} 
In this experiment, we  compare PPNNs with  Haar wavelet thresholding both for discrete Haar bases and Haar frames arising from the undecimated (translation invariant) version of the Haar transform.
In particular, the experiment is linked to the starting point of our considerations, namely wavelet and frame shrinkage.
For further details on the corresponding filter banks we refer to \cite{CV98,PS2006}.
As quality measure for our experiments we choose the average peak signal to noise ratio (PSNR) over the test set.
Recall that for a prediction $x\in\R^{n}$ and ground truth $y\in\R^{n}$ the PSNR is defined by
$$
\mathrm{PSNR}(x,y)=10\log_{10}\Big(\frac{(\max y - \min y)^2}{\sum_{i=1}^{n} (x_i-y_i)^2}\Big).
$$

Since we focus on the Haar filter, we restrict our attention to piecewise constant signals with mean $0$.
By $(x_i,y_i)\in\R^{d}\times\R^{d}$, $i=1,\ldots,N$, we denote pairs of piecewise constant signals $y_i$ of length $d=2^7=128$ and their noisy versions by $x_i=y_i+\epsilon_i$, 
where $\epsilon_i$ is white noise with standard deviation $\sigma = 0.1$.
For the \emph{signal generation}, we choose 
\begin{itemize}
\item  the number of constant parts of $y_i$ as $\max\{2,t_i\}$, where $t_i$ is the realization of a random variable 
following the Poisson distribution with mean $5$;
\item the discontinuities of $y_i$ as realization of a uniform distribution;
\item the signal intensity of $y_i$ for every constant part as realization of the standard normal distribution, where we subtract the mean of the signal finally.
\end{itemize}
Using this procedure, we generate training data $(x_i,y_i)_{i=1}^N$ and test data $(x_i,y_i)_{i=N+1}^{N+N_\text{test}}$ 
with $N=500000$ and $N_\text{test}=1000$.
The average PSNR of the noisy signals in the test set is $25.22$.
We use PPNNs with $K-1 \in \{1, 2, 3\}$ layers and set $T_K=I_d$ and $b_K=0$.
In all examples a batch size of $32$ and a learning rate of $0.5$ is used.
\medskip

We are interested in two different settings:

1. \emph{Learned orthogonal matrices versus Haar basis.} 
First, we consider PPNNs with $128$~neurons in each hidden layer and
componentwise soft-shrinkage $S_\lambda$ as activation function.
In particular, all matrices $T_k$ have to be orthogonal.
The denoising results of our learned PPNN are compared with the soft wavelet shrinkage with respect 
to the discrete orthogonal Haar basis in $\mathbb R^{128}$, i.e., the signal on all $6$ scales is decomposed by 
\begin{align}\label{Haar-shrinkage}
\Psi(x) = H^*  S_{\lambda}(H x),
\end{align}
where
$
H \coloneqq H_2 \, \cdots \, H_7
$
with matrices
\begin{equation}
H_{j} \coloneqq 
\begin{pmatrix} 
\tilde H_j&0\\
0&I_{2^7 - 2^{j}}
\end{pmatrix}, 
\quad
\tilde H_j \coloneqq \frac{1}{\sqrt{2}}
\mbox{\scriptsize
$\begin{pmatrix}
1&1&0&0&\ldots&0&0&0&0\\
0&0&1&1&\ldots&0&0&0&0\\
\vdots\\
0&0&0&0&\ldots&1&1&0&0\\
0&0&0&0&\ldots&0&0&1&1\\
1&-1&0&0&\ldots&0&0&0&0\\
0&0&1&-1&\ldots&0&0&0&0\\
\vdots\\
0&0&0&0&\ldots&1&-1&0&0\\
0&0&0&0&\ldots&0&0&1&-1\\
\end{pmatrix}$}
\in \mathbb R^{2^j,2^j}.
\end{equation}
The  average PSNRs on the test data are given in Tab.~\ref{tab:orth}.
For determining the optimal threshold in $S_\lambda$, we implemented two different methods. The first one is $5$-fold cross validation (CV).
More precisely, the training data is divided into $5$ subsets and each is used once as a test set with the remaining samples as training set.
The test loss for given $\lambda$ is averaged over all $5$ trials for judging the quality of the model.
The tested parameters $\lambda$ are chosen in $[0.05, 0.3]$ with steps of $0.05$ for a NN with one layer.
The second method is to set $\lambda$ as a trainable variable of the neural network and optimize it via stochastic gradient descent (SGD) during the training process.
For NNs with two and three layers, the tested parameters are divided by $2$ and $3$, respectively.
It appears that for only one hidden layer the Haar wavelet shrinkage is still better than the learned orthogonal matrix.
If we increase the number of layers, then PPNNs lead to a better average PSNR.

Two exemplary noisy signals and their denoised versions are shown in Fig.~\ref{fig:signals_orth}.
Since we have learned the orthogonal matrices $T_k$ with respect to the quadratic loss function, the visual quality of the PPNN denoised signals is clearly not satisfactory even with an improved PSNR.
The visual impression of signals denoised by Haar wavelet shrinkage can by improved (smoother signal) by increasing the threshold to $\lambda = 0.3$, resulting in a worse PSNR. 
To achieve a similar behavior with orthogonal matrices learned by PPNNs, we have to choose a different loss function.

\begin{table}
\begin{center}
\begin{tabular}{c|ccc|ccc}
&&CV&&&SGD&\\
Method& PSNR & Loss & Optimal $\lambda$& PSNR & Loss & Optimal $\lambda$\\\hline
Haar basis&$29.89$&$0.00359$&$0.1$&$29.99$&$0.00355$&$0.109$\\
One layer &$29.73$&$0.00373$&$0.1$&$29.74$&$0.00376$&$0.106$\\
Two layers&$30.55$&$0.00313$&$0.05$&$30.85$&$0.00299$&$0.0583$\\
Three layers&$30.83$&$0.00296$&$0.033$&$31.13$&$0.00284$&$0.0164$
\end{tabular}
\end{center}
\caption{PSNRs (average on test data) for denoising piecewise constant signals.
For only one layer, the Haar wavelet shrinkage is still better than the learned orthogonal matrix.
This changes if we increase the number of layers. CV denotes cross validation, SGD the learned threshold.}
\label{tab:orth}
\end{table}

\begin{figure}
\begin{subfigure}[t]{0.32\textwidth}
\centering
\includegraphics[width=\textwidth]{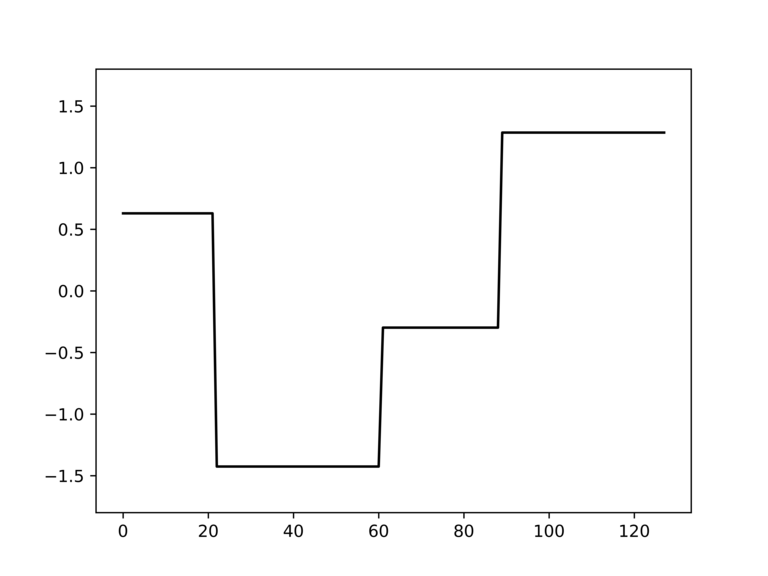}
\caption{Original I}
\end{subfigure}
\hfill
\begin{subfigure}[t]{0.32\textwidth}
\centering
\includegraphics[width=\textwidth]{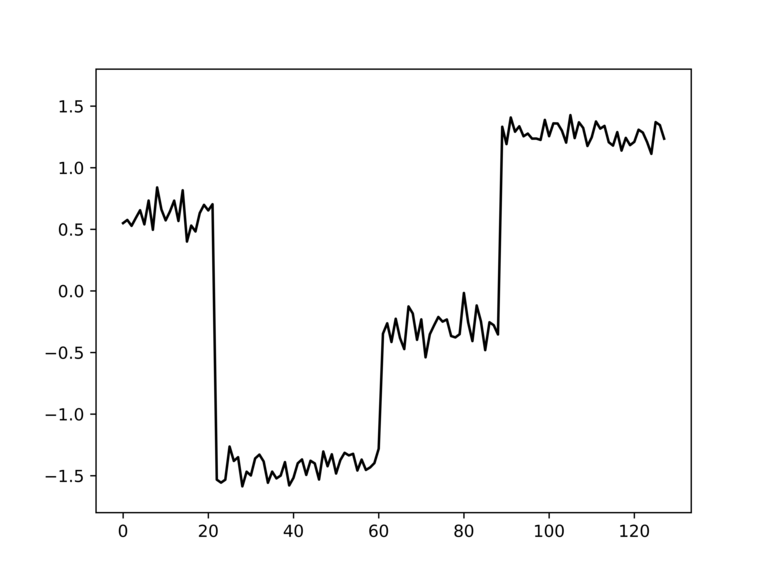}
\caption{Noisy (PSNR 29.11)}
\end{subfigure}\hfill
\begin{subfigure}[t]{0.32\textwidth}
\centering
\includegraphics[width=\textwidth]{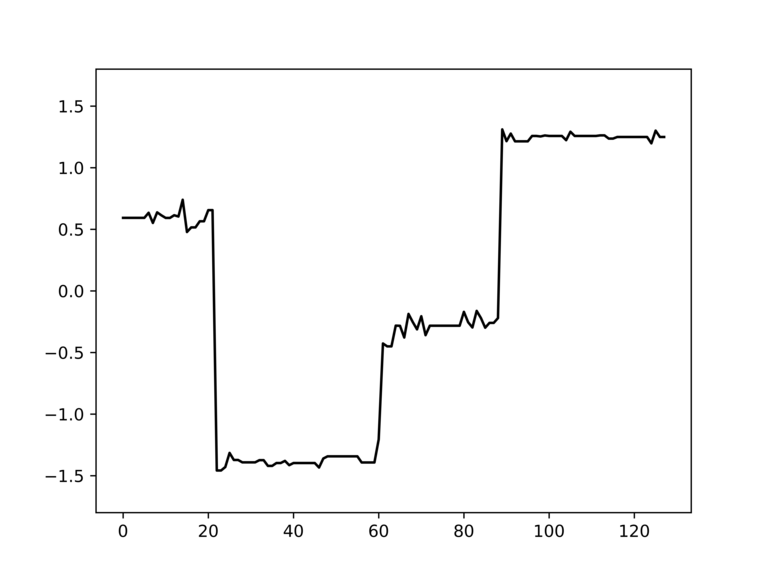}
\caption{Haar  basis, $\lambda=0.109$\\(PSNR 33.14)}
\end{subfigure}
\\
\begin{subfigure}[t]{0.32\textwidth}
\centering
\includegraphics[width=\textwidth]{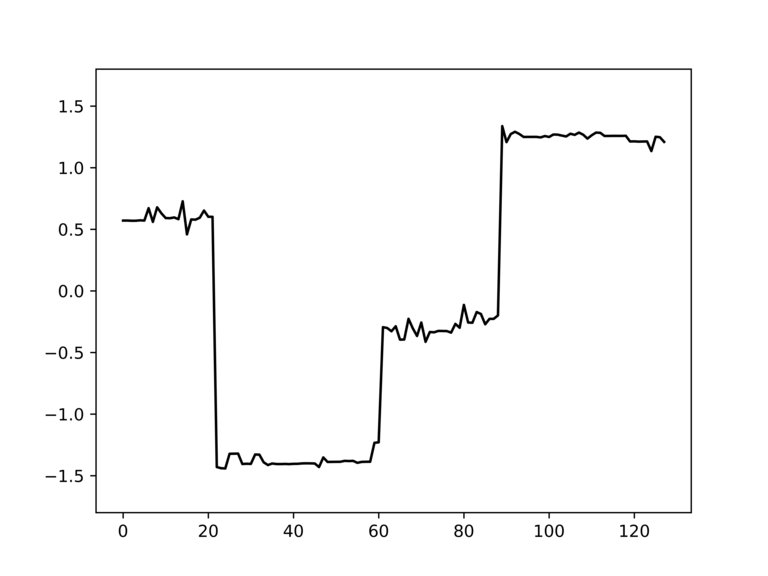}
\caption{1 layer, $\lambda=0.106$\\(PSNR 33.06)}
\end{subfigure}\hfill
\begin{subfigure}[t]{0.32\textwidth}
\centering
\includegraphics[width=\textwidth]{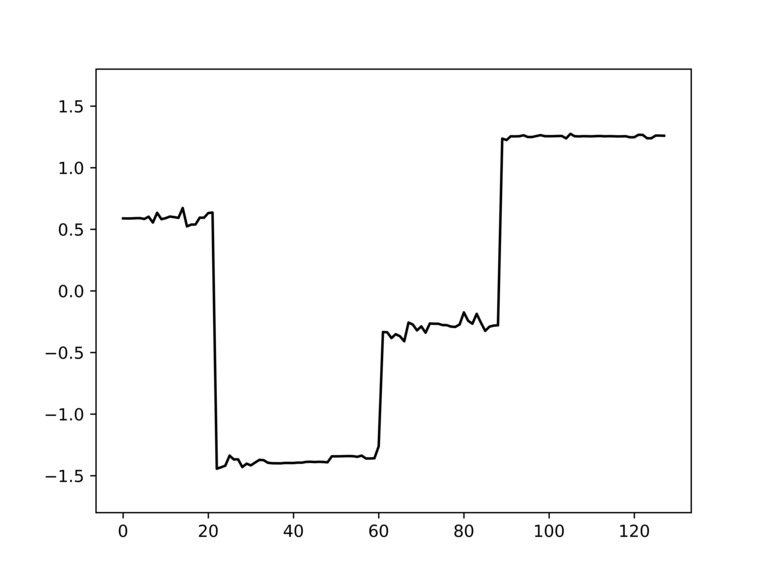}
\caption{3 layers, $\lambda=0.0390$\\ (PSNR 34.90)}
\end{subfigure}\hfill
\begin{subfigure}[t]{0.32\textwidth}
\centering
\includegraphics[width=\textwidth]{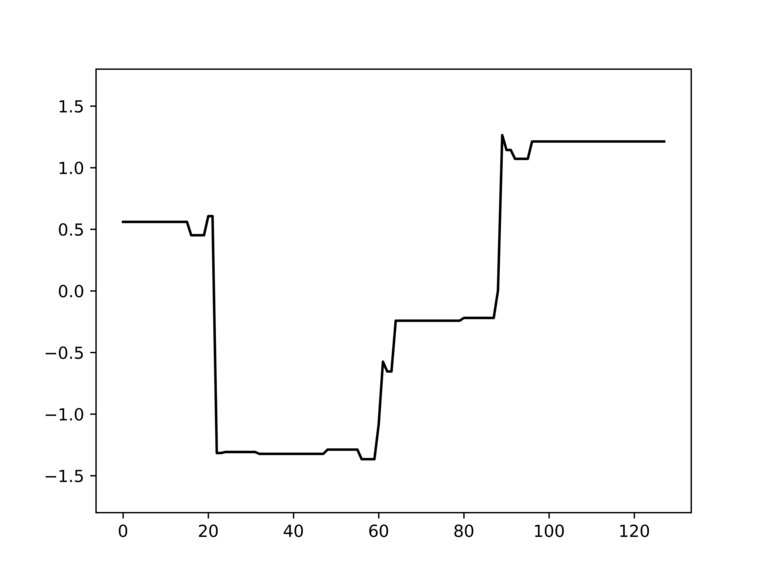}
\caption{Haar  basis, $\lambda=0.3$\\ (PSNR 27.43)}
\end{subfigure}\hfill
\\
\begin{subfigure}[t]{0.32\textwidth}
\centering
\includegraphics[width=\textwidth]{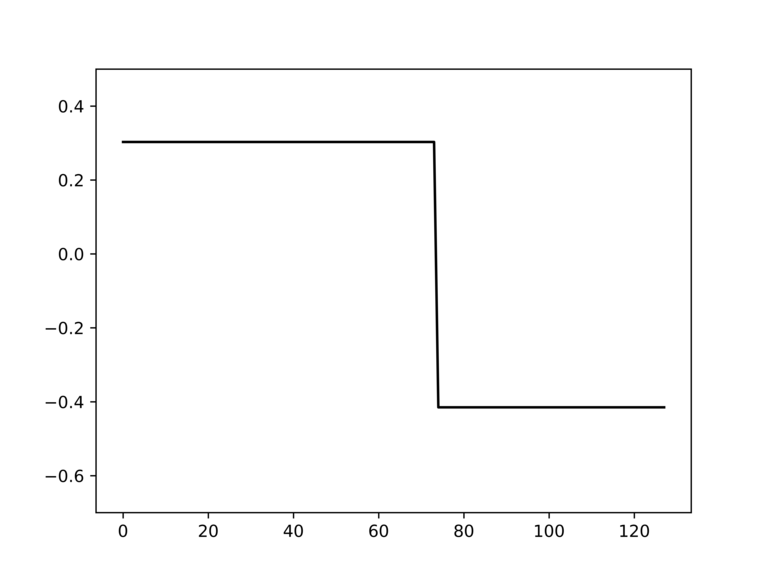}
\caption{Original II}
\end{subfigure}\hfill
\begin{subfigure}[t]{0.32\textwidth}
\centering
\includegraphics[width=\textwidth]{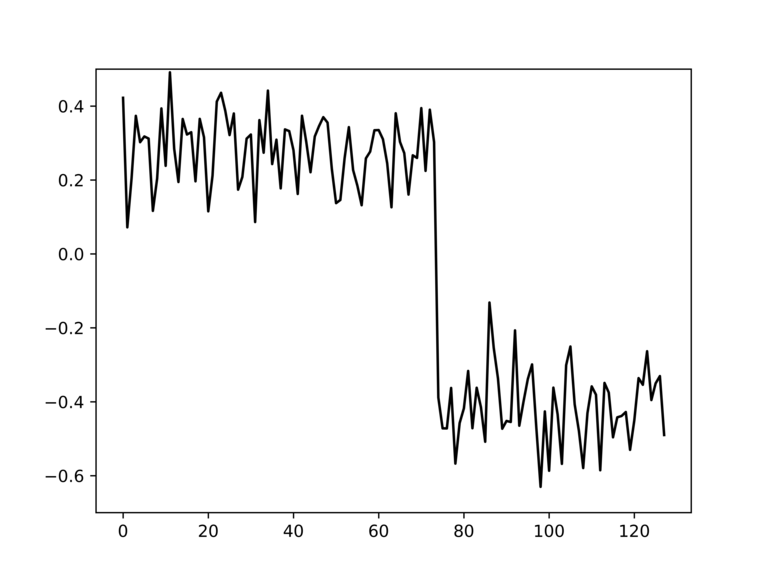}
\caption{Noisy (PSNR 17.45)}
\end{subfigure}\hfill
\begin{subfigure}[t]{0.32\textwidth}
\centering
\includegraphics[width=\textwidth]{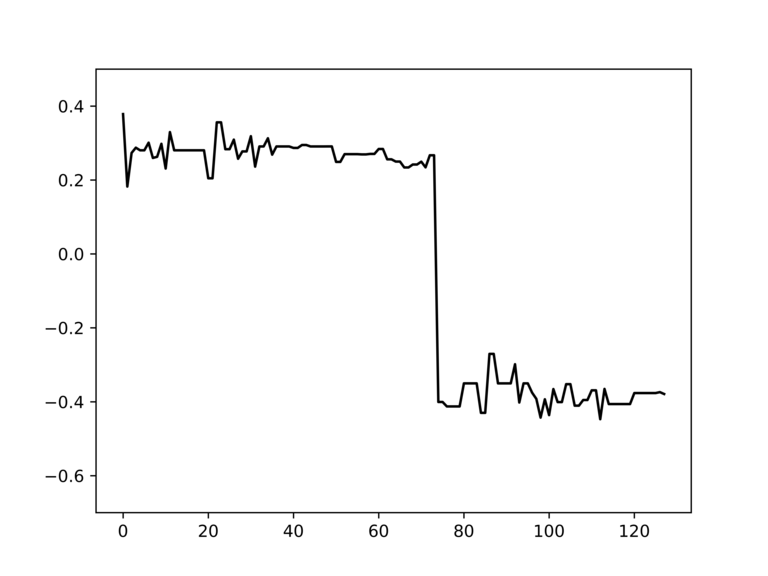}
\caption{Haar basis, $\lambda=0.109$\\(PSNR 24.05)}
\end{subfigure}\hfill
\begin{subfigure}[t]{0.32\textwidth}
\centering
\includegraphics[width=\textwidth]{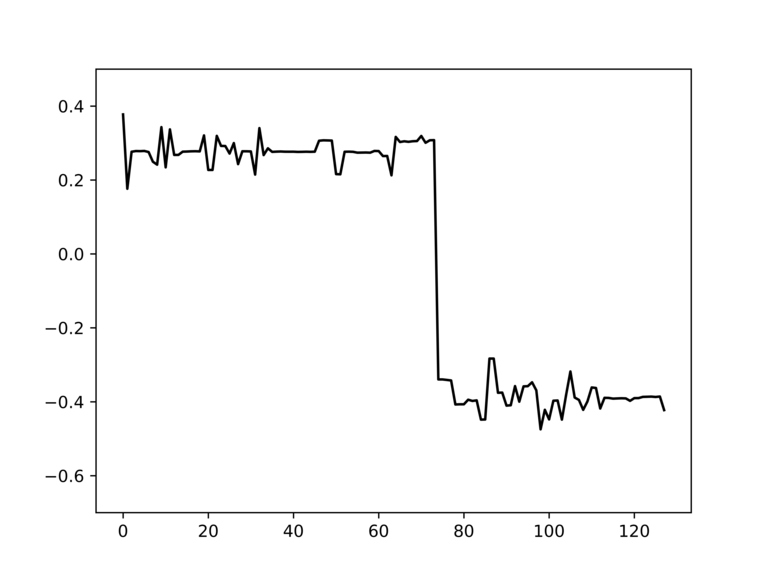}
\caption{1 layer, $\lambda=0.106$\\ (PSNR 24.60)}
\end{subfigure}
\hfill
\begin{subfigure}[t]{0.32\textwidth}
\centering
\includegraphics[width=\textwidth]{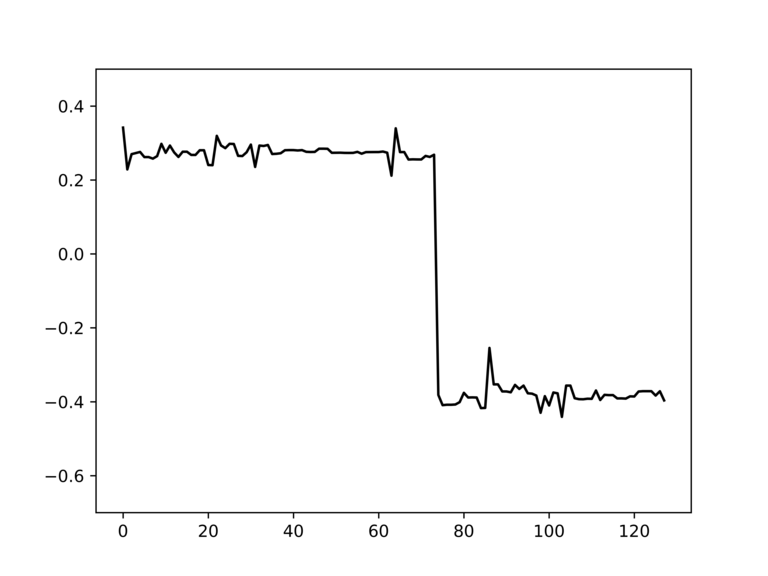}
\caption{3 layers, $\lambda=0.0390$\\ (PSNR 25.68)}
\end{subfigure}\hfill
\begin{subfigure}[t]{0.32\textwidth}
\centering
\includegraphics[width=\textwidth]{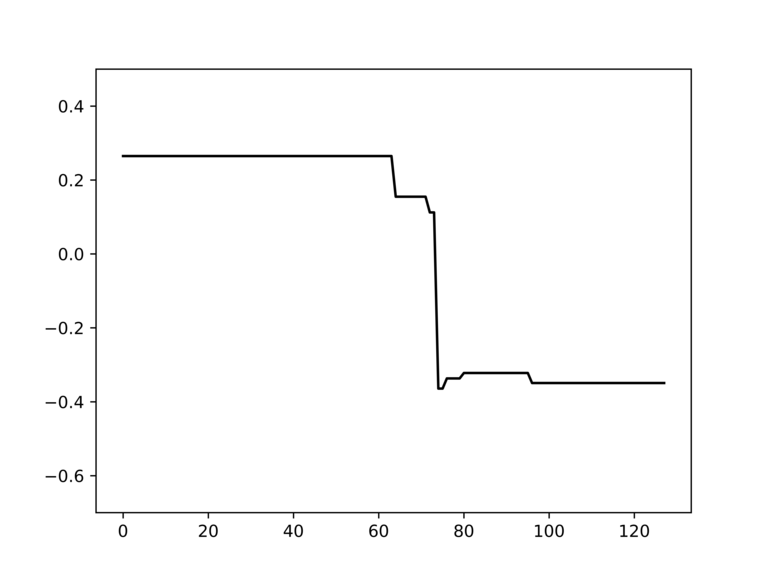}
\caption{Haar  basis, $\lambda=0.3$\\ (PSNR 20.09)}
\end{subfigure}
\caption{Two denoising examples using the Haar basis and learned orthogonal matrices. 
The signals denoised by PPNNs look better for an increasing number of layers.
For Haar wavelet shrinkage, a smoother denoised signal can be attained 
by increasing the threshold to $\lambda = 0.3$, although this signal has a smaller PSNR.}
\label{fig:signals_orth}
\end{figure}
\medskip

2. \emph{Learned Stiefel matrices versus Haar frame.}
Haar wavelet shrinkage can be improved by using Haar wavelet frames within a so-called ``algorithm \'a trous'', see \cite{mallat2008}. 
We apply a similar method as in  \eqref{Haar-shrinkage}, but with a rectangular matrix $H$ whose rows form a Haar frame.
More precisely, the Haar filter is used without subsampling. 
This results, in contrast to the original Haar transform, in a translational invariant multiscale transform.
Instead of the matrices $H_7$,  the nonsubsampled (convolution) matrix
\begin{equation}\label{nolanted} \mbox{\scriptsize $\frac{1}{\sqrt{2}}
\begin{pmatrix}
1&1&0&0&\ldots&0&0&0&0\\
0&1&1&0&\ldots&0&0&0&0\\
\vdots\\
0&0&0&0&\ldots&&1&1&0\\
0&0&0&0&\ldots&0&0&1&1\\
1&0&0&0&\ldots&0&0&0&1\\
1&-1&0&0&\ldots&0&0&0&0\\
0&1&-1&0&\ldots&0&0&0&0\\
\vdots\\
0&0&0&0&\ldots&0&1&-1&0\\
0&0&0&0&\ldots&0&0&1&-1\\
-1&0&0&0&\ldots&0&0&0&1\\
\end{pmatrix} 
\in \mathbb R^{256,128}$.}
\end{equation}
is applied to the original signal. Note that we assume a periodic continuation of the signal.
In each of the following $j$ steps, $j=1, \ldots,6$, we keep the lower part and transform again the upper smoothed part by essentially the same matrix, where
$2^{j}-1$ zeros are inserted between the filter coefficients. 
The output signal has size $8\cdot128$, where the last part of the signal is just the averaged signal, which is equal to zero.
Overall, the original signal  is multiplied by
$H\in \R^{1024, 128}$, where for $j\in\{0,\ldots,6\}$ and $i\in \{0,\ldots,128-2j\}$ the $(i+128 j)$-th row is given by one element of a Haar wavelet frame
\begin{align}
(\underbrace{\begin{array}{ccc}0&...&0 \end{array}}_{i \text{ times } 0}
\underbrace{\begin{array}{ccc}1&...&1 \end{array}}_{j \text{ times } 1}
\underbrace{\begin{array}{ccc}-1&...&-1 \end{array}}_{j \text{ times } -1}
\underbrace{\begin{array}{ccc}0&...&0 \end{array}}_{128-i-2j \text{ times } 0})
\end{align}
and the last $128$ rows of $H$ are given by $(\begin{array}{ccc}1&...&1\end{array})$. 
It is well known that for the above translation invariant Haar frame transform,
a scale dependent shrinkage has to be applied to each scale, namely starting with threshold $\lambda$ the next scales should be thresholded by
\begin{align}\label{threshold_adapt}
\frac{1}{\sqrt{2}^{j}}\lambda, \quad j=0,\ldots,6.
\end{align}
For an explanation of this statement we refer to \cite{SWPMW2004}.
In summary, we obtain
\begin{align}
\hat\Psi(x)=H ^*\tilde S_\lambda(H x),
\end{align}
where $\tilde S_\lambda$ denotes the scale-wise adapted thresholding.

Now, we compare this scale-dependent Haar frame soft thresholding method with a learned PPNN with $1024$ neurons in the hidden layers and componentwise soft-shrinkage $S_\lambda$ as activation function.
The optimal threshold in $S_\lambda$ is again determined either by $3$-fold cross validation, where the tested parameters are chosen in $[0.01, 0.1]$ with steps of $0.01$, or by setting $\lambda$ as a trainable variable of the neural network and optimize it via stochastic gradient descent.
We emphasize that in contrast to the Haar frame shrinkage procedure with~\eqref{threshold_adapt}, the same threshold for each component is used in the activation function of our PPNN.
The resulting PSNRs are given in Tab.~\ref{tab:haar_NN}.
As expected, using the same threshold in the classical Haar frame shrinkage is worse than the scale-adapted Haar frame shrinkage.
PPNNs with learned Stiefel matrices perform better for an increasing number of layers.

Finally, Fig.~\ref{fig:signals_frame} contains the denoised signals from Fig.~\ref{fig:signals_orth}.
The results are visually better than in the previous figure, although still not satisfactory due to the used loss function.

\begin{table}[t]
\centering
\begin{tabular}{c|ccc|ccc}
&&CV&&&SGD&\\
Method & PSNR & Loss & Optimal $\lambda$ & PSNR & Loss & Optimal $\lambda$\\\hline
Haar frame with $S_\lambda$& $26.99$& $0.00669$ & $0.02$ & $27.14$ & $0.00649$&$0.0265$\\
Haar frame with $\tilde S_\lambda$& $30.58$& $0.00307$ & $0.08$ & $30.59$&$0.00307$&$0.0820$\\
One layer PPNN & $32.46$ & $0.00211$ & $0.06$ & $32.50$ & $0.00207$ & $0.0514$\\
Two layer PPNN & $32.79$ & $0.00200$ & $0.04$ & $33.05$ & $0.00186$ & $0.0250$\\
Three layer PPNN & $33.13$ & $0.00185$ & $0.02$ & $33.22$ & $0.00181$ & $0.0164$ 
\end{tabular}
\caption{PSNRs (average on test data) for denoising piecewise constant signals.
The first row with  non-scale-adapted thresholds are worse than those with the adapted ones in the second row.
Learned Stiefel matrices with the same componentwise activation function
lead to better results when the corresponding PPNNs are applied for denoising. CV denotes cross validation, SGD the learned threshold.
}
\label{tab:haar_NN}
\end{table}

\begin{figure}[!t]
\begin{subfigure}[t]{0.325\textwidth}
\centering
\includegraphics[width=\textwidth]{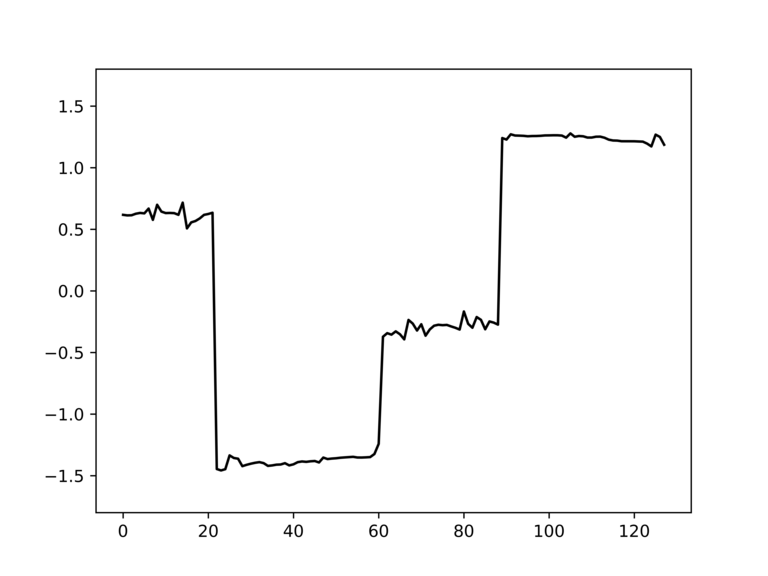}
\caption{Haar frame with $\tilde S_\lambda$,\\$\lambda=0.082$, (PSNR $34.25$)}
\end{subfigure}
\hfill
\begin{subfigure}[t]{0.325\textwidth}
\centering
\includegraphics[width=\textwidth]{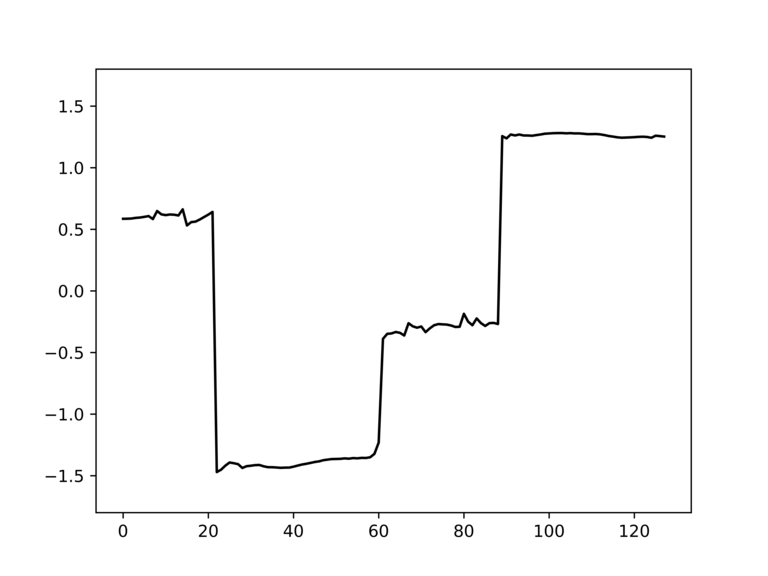}
\caption{1 layer, $\lambda=0.0515$,\\(PSNR $36.24$)}
\end{subfigure}
\hfill
\begin{subfigure}[t]{0.325\textwidth}
\centering
\includegraphics[width=\textwidth]{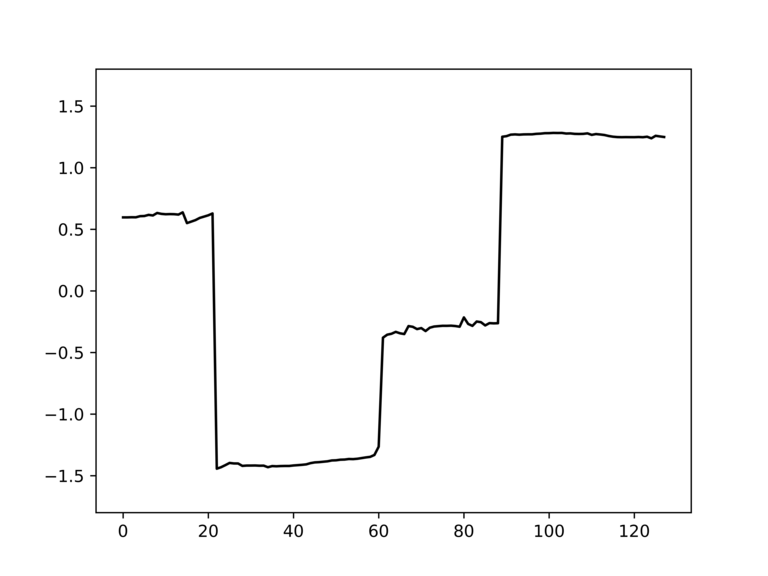}
\caption{3 layers, $\lambda=0.0164$,\\(PSNR $37.53$)}
\end{subfigure}
\hfill
\begin{subfigure}[t]{0.325\textwidth}
\centering
\includegraphics[width=\textwidth]{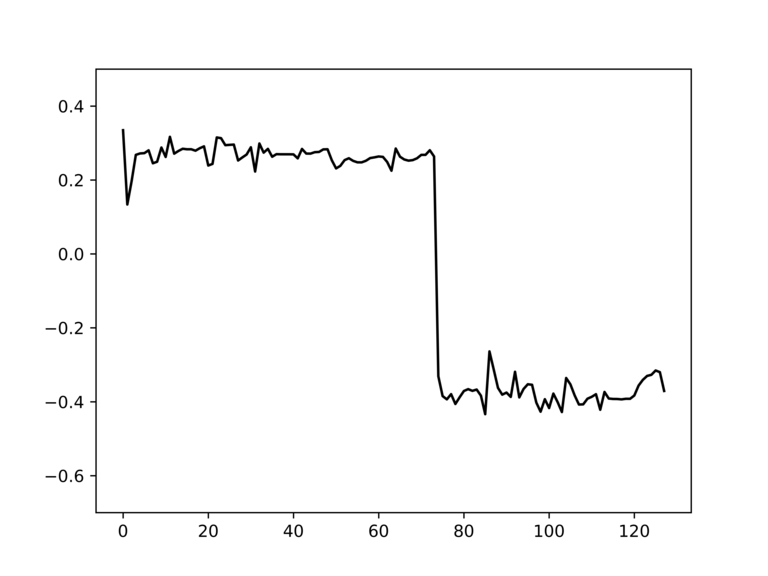}
\caption{Haar frame with $\tilde S_\lambda$,\\$\lambda=0.082$, (PSNR $23.46$)}
\end{subfigure}
\hfill
\begin{subfigure}[t]{0.325\textwidth}
\centering
\includegraphics[width=\textwidth]{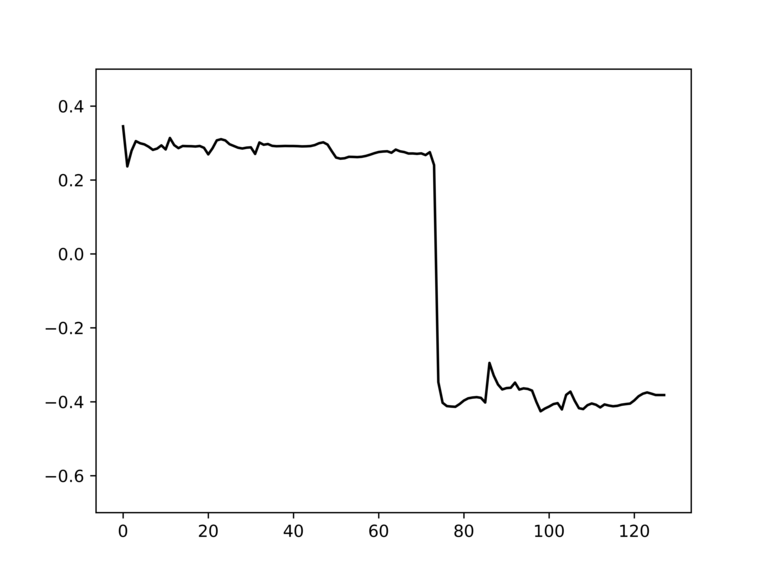}
\caption{1 layer, $\lambda=0.0515$,\\(PSNR $27.60$)}
\end{subfigure}
\hfill
\begin{subfigure}[t]{0.325\textwidth}
\centering
\includegraphics[width=\textwidth]{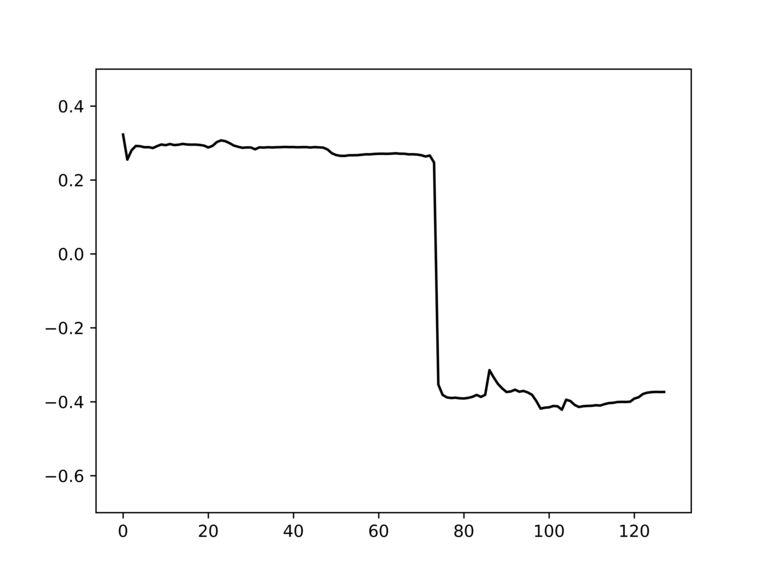}
\caption{3 layers, $\lambda=0.0164$,\\(PSNR $28.06$)}
\end{subfigure}

\caption{Denoising for the signals in Fig.~\ref{fig:signals_orth}.
The undecimated Haar frame with scale-adapted shrinkage 
and learned Stiefel matrices of the same size as the Haar frame are compared for $\lambda$ from Tab.~\ref{tab:haar_NN}.
The PPNN denoised signals are visually nicer than those with scale-adapted Haar frame shrinkage.}
\label{fig:signals_frame}
\end{figure}
\medskip

\noindent
\textbf{Classification:} In this example, we train a PPNN for classifying the MNIST data set\footnote{http://yann.lecun.com/exdb/mnist}.
The length of the input signals is $d=28^2$.
We consider a PPNN with $K-1 =5$ layers and $n_1=n_2=784$, $n_3=n_4=400$ and $n_5=200$ neurons in the layers and componentwise applied ReLu activation function $\sigma(x) = \max(0,x)$.
To get 10 output elements (probabilities)  in $(0,1)$, we use an additional sixth layer 
$$
g (T_K x + b_K), \quad T_K \in \mathbb R^{10,d}, \, b_K \in \mathbb R^{10}
$$
with another activation function $g(x) \coloneqq \frac{1}{1 + \mathrm{exp} (-x)}$.
For training, we use a batch size of $1024$ and a learning rate of $5$.
After $1000$ epochs we reach an accuracy of $0.9855$ on the test set. One epoch takes about one second on a NVIDIA Quadro M5000 GPU.
In Fig.~\ref{fig:training}, the training and test loss of our PPNN during training are plotted. 
\begin{remark}
As already mentioned in Remark \ref{OMDSM},
NNs with Stiefel matrices were also applied in \cite{huang2018orthogonal}.
The authors of \cite{huang2018orthogonal} reported that the training process using Riemannian optimization 
on the Stiefel manifold could be unstable or divergent.
We do not observe such instabilities in our setting.
\end{remark}

\begin{figure}[!t]
\centering
\includegraphics[width=0.5\textwidth]{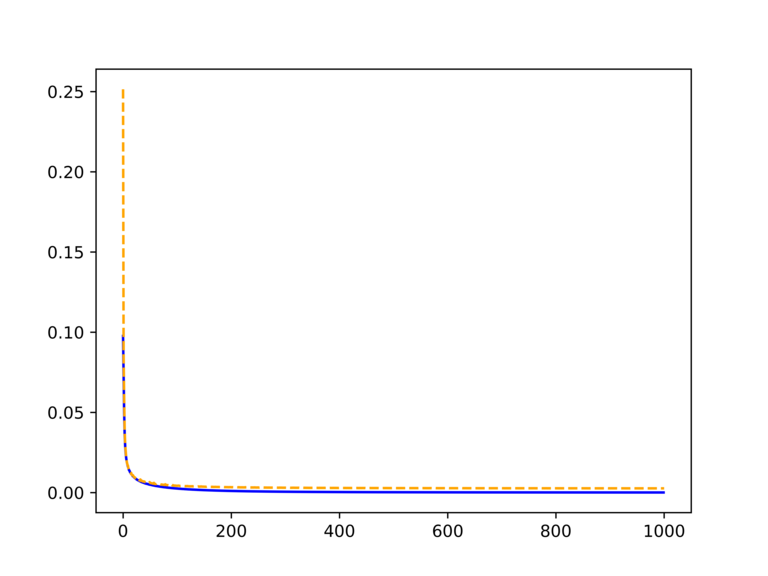}
\caption{Training loss (blue) and test loss (orange) of a PPNN on the MNIST data set. 
The x-axis corresponds to the number of epochs and the y-axis to the associated loss value.}
\label{fig:training}
\end{figure}

\noindent
\textbf{Adversarial Attacks:} Neural networks with bounded Lipschitz constants 
were successfully applied to defend against adversarial attacks, see \cite{GSS2015,TSS2018}. 
In this example, we demonstrate that a PPNN is more robust under adversarial attacks than a standard neural network.
Assume that we have a neural network 
$f=(f_1,\ldots,f_{10})\colon\R^{28^2}\to(0,1)^{10}$ for classifying MNIST, 
e.g., a PPNN as described in the previous example. 
Further, we have given an input $x\in\R^{28^2}$. Now, we perform an adversarial attack in the following way:
\begin{itemize}
\item Set $\nu \coloneqq\argmax_{i\in\{1,\ldots,10\}} f_i(x)$, and $g \coloneqq \nabla_x \tfrac{f_\nu (x)}{\|f(x)\|_1}$.
\item Initialize $\epsilon = 10^{-2}$ and while $\nu=\argmax_{i\in\{1,\ldots,10\}} f_i(x-\epsilon g)$ update $\epsilon=2\epsilon$.
\end{itemize}
This procedure is applied on two neural networks. 
More precisely, the first one is the PPNN from the previous example and the second one is a neural network with the same structure as the PPNN, but without the orthogonality constraint.
We train the standard neural network using the Adam optimizer with learning rate of $10^{-4}$ and end up with an accuracy of $0.9863$. 
Then, we perform an adversarial attack on both of these networks and record the norm $\|\epsilon g\|_2$ of the noise that changes the prediction.
We perform this for each input $x_k\in[0,255]^{28^2}$, $k=1,\ldots,10000$,
in the test set and compute the mean, standard deviation and median of these norms. 
For the PPNN, we record an average norm of $38.28\pm24.51$ and a median of $33.71$. 
For the standard neural network, 
we record an average norm of $30.48\pm 15.72$ and a median of $28.75$. 
Overall, the PPNN seems to be more stable against such adversarial attacks.

\section{Conclusions} \label{sec:conclusions}
In this paper, we have shown that for real Hilbert spaces $\HH$ and $\KK$, a proximity operator $\Prox\colon \KK \to \KK$  
and a linear bounded operator $T\colon \HH \to \KK$ 
the operator $T^\dagger\, \Prox (T \cdot + b)$ with $b \in \KK$ is a proximity operator on $\HH$.
As a consequence, the famous frame soft shrinkage operator can be seen  as a proximity operator. 
Using this new relations, we have discussed special neural networks arising from Parseval frames and stable activation functions.
Our networks are Lipschitz networks, which are moreover averaged operators.
These networks include recently proposed ones containing matrices whose transposes are in a Stiefel manifold
and interpret them from another, more general point of view.

In our future work, we want to explore for which learning tasks the higher flexibility of our PPNNs
is advantageous. Taking  more general operators $T$ into account may be also useful.
In particular, we will apply our PNNs  within Plug-and-Play algorithms.
Another question that we want to address is to constrain our Stiefel matrices further, e.g., towards convolutional networks and to sparsity constraints. 
Depending on the application, we have to design appropriate loss functions as well as incorporating regularizing terms.

For our experiments the stochastic gradient algorithm on Stiefel manifolds worked well.
However, other minimization methods could be 
taken into account.
In \cite{huang2018orthogonal} for example, the authors proposed an orthogonal weight normalization algorithm  that was inspired by the fact that eigenvalue decomposition is differentiable. 
Finally, we like to mention that a proximal backpropagation algorithm 
taking implicit instead of explicit gradient steps to update the network parameters during neural network training 
was proposed in \cite{FMMC2017}.

A better understanding of the convergence of the cyclic proximal point algorithm, see Remark \ref{rem:cpp}, 
and suitable early stopping criteria if the network $\Phi$ is iteratively used may help to design NNs and to understand their success.
  
\appendix
\section{Gradient computation}
\begin{proof}[\textbf{Proof of Lemma \ref{grad}}:]
For potential further applications, we compute the gradient of 
the more general functional 
\[J(T,b) \coloneqq \ell \bigl(T^\dagger \sigma(Tx+b);y\bigr) = \ell \bigl( f(T,b);y \bigr),\]
where $T^\dagger = (T^* T)^{-1} T^*$. 
Note that the case $T^\dagger = T^* (T T^* )^{-1}$ can be treated in an analog way.
Using
$f_1(T) = (T^* T) ^{-1}$,  
$f_2(T) = T^*$,
$f_3(T,b) = \sigma(Tx+b)$, the function $f(T,b) \coloneqq T^\dagger \sigma(Tx + b)$ is decomposed as 
$$
f(T,b)   = f_1(T) \, f_2(T) \, f_3(T,b).
$$
By the product rule it holds
\begin{align}
	&D_T f(T,b)[H]\\ = &Df_1(T)[H] \, f_2(T) \, f_3(T,b) + f_1(T) \,  Df_2(T)[H] \, f_3(T,b) + f_1(T) \, f_2(T) \,  D_Tf_3(T,b)[H].
\end{align}
The involved differentials are given by
\begin{align}
  D f_1(T)[H] &= -(T^*T )^{-1} (T^* H + H^* T) (T^*T )^{-1},\\
  D f_2(T)[H] &= H^*,\\
  D_Tf_3(T,b)[H] &= \mathrm{ diag} \bigl(\sigma'(Tx+b)\bigr) \,  Hx.
\end{align}
Consequently, we obtain
\begin{align}
 D_T f(T,b)[H] = &
 -(T^* T)^{-1} \bigl((T^* H + H^* T) (T^* T)^{-1} s
  + H^* \, r 
 + T^* \Sigma H x\bigr).
\end{align}
Using the chain rule, we conclude
\begin{align}
&D_T J(T,b)[H] = \langle \nabla_T J(T,b) , H \rangle 	= \langle t, D_T f(T,b)[H] \rangle\\
=& - \langle t, (T^* T)^{-1} (T^* H + H^* T) (T^* T)^{-1} s \rangle + \langle t, (T^* T) ^{-1}  H^* r\rangle + \langle t,(T^* T) ^{-1}T^* \Sigma Hx\rangle\\
=& - \langle T (T^* T)^{-1}t s^* (T^* T)^{-1},  H \rangle - \langle (T^* T)^{-1} t s^* (T^* T)^{-1} T^*,  H^*\rangle\\
& \quad +\langle (T^* T) ^{-1} t r^*,   H^* \rangle + \langle \Sigma \, T (T^* T) ^{-1} t x^*, H \rangle \\
=& \bigl\langle -T (T^* T)^{-1}(t s^* + s t^*)(T^* T)^{-1} + r t^* (T^* T) ^{-1}    + \Sigma \, T (T^* T) ^{-1} t x^*, H \bigr\rangle.
\end{align}
Thus,
\begin{align}
\nabla_T J(T,b) 
&= -T (T^* T)^{-1} \left( t s^* + s t^* \right) (T^* T)^{-1}
 + r t^* (T^* T) ^{-1}    + \Sigma \, T (T^* T) ^{-1} t x^*.
\end{align}
For the gradient with respect to $b$, we obtain by the chain rule
\begin{align}
\nabla_b J(T,b)^*
&= t^* T^\dagger \nabla_b \bigl(\sigma(Tx+b)\bigr)(b)= t^* T^\dagger \Sigma.
\end{align}
\end{proof}

\section{Activation functions}\label{sec:activation}

The following table lists many functions $f$ having a proximity $\sigma$ that is a common activation function
in NNs from \cite{CP2018}.

{\scriptsize
\begin{table}
{\scriptsize
\begin{center}
\begin{tabular}{| c | c | c | }
\hline
 Name & activation function $\sigma(x)$ &  $f(x)$ with $\sigma = \prox_{f}$
\\
\hline
\hline
     Linear activation & $x$ & $0$
		\\[0.5ex]
		\hline
		\shortstack[c]{Saturated linear \\ activation (SaLU)} & 
		$		
    \begin{cases}
		1 & \text{if\quad$x>1$} \\
        x & \text{if\quad$-1\leq x\leq 1$}\\
        -1 & \text{if\quad$x<-1$}
  \end{cases}
	$ 
	& $\iota_{[-1,1]}=
  \begin{cases}
		0 & \text{if\quad$x\in [-1,1]$} \\
        \infty & \text{if\quad$x\notin [-1,1]$}\\
  \end{cases}$
    \\[0.5ex]\hline
  Soft Thresholding & $\begin{cases}
		x-\lambda & \text{if\quad$x> \lambda$} \\
		0 & \text{if\quad$x \in [-\lambda,\lambda]$}\\
                x+\lambda & \text{if\quad$x < -\lambda$}  \\
  \end{cases}$ & $\lambda |x|$
		\\[0.5ex]
		\hline
		\shortstack[c]{Saturated linear \\ activation (SaLU)} &  
		$		
    \begin{cases}
		1 & \text{if\quad$x>1$} \\
        x & \text{if\quad$-1\leq x\leq 1$}\\
        -1 & \text{if\quad$x<-1$}
  \end{cases}
	$ 
	& $\iota_{[-1,1]}=
  \begin{cases}
		0 & \text{if\quad$x\in [-1,1]$} \\
        \infty & \text{if\quad$x\notin [-1,1]$}\\
  \end{cases}$
    \\[0.5ex]\hline

  \shortstack[c]{Rectified linear unit\\ (ReLU)} & 
	$\begin{cases}
		x & \text{if\quad$x>0$} \\
        0 & \text{if\quad$x\leq 0$}\\
  \end{cases}$ & $\iota_{[0,\infty)}=
  \begin{cases}
		0 & \text{if\quad$x\in [0,\infty)$} \\
        \infty & \text{if\quad$x\notin [0,\infty)$}\\
  \end{cases}$
  \\[0.5ex]\hline
	
  \shortstack[c]{Parametric rectified \\linear unit (PReLU)} &  $
    \begin{cases}
		x & \text{if\quad$x>0$} \\
        \alpha x & \text{if\quad$x\leq 0$}\\
  \end{cases}$ , $\alpha\in (0,1]$ & $
   \begin{cases}
		0 & \text{if\quad$x>0$} \\
        (\frac{1}{\alpha}-1)\dfrac{x^2}{2} & \text{if\quad$x\leq 0$}\\
  \end{cases}$
  \\[0.5ex]\hline
	
  \shortstack[c]{Bent identity\\ activation} & $ \dfrac{x+\sqrt{x^2+1}}{2}$ & $
  \begin{cases}
		x/2-\ln(x+\frac{1}{2})/4 & \text{if\quad$x>-\dfrac{1}{2}$} \\
        \infty & \text{if\quad$x\leq -\dfrac{1}{2}$}\\
  \end{cases}$
	\\[0.5ex]\hline
  
  \shortstack[c]{Inverse square root\\linear unit} & $
    \begin{cases}
		x & \text{if\quad$x\geq 0$} \\
        \dfrac{x}{\sqrt{x^2+1}} & \text{if\quad$x<0$}\\
  \end{cases}$ & $
  \begin{cases}
  		0 & \text{if\quad$x\geq 0$} \\
		1-x^2/2-\sqrt{1-x^2} & \text{if\quad$-1\leq x<0$} \\
        \infty & \text{if\quad$x<-1$}\\
  \end{cases}$
	\\[0.5ex]\hline
  
   \shortstack[c]{Inverse square\\ root unit} & $
        \dfrac{x}{\sqrt{x^2+1}}$ & $
  \begin{cases}
		-x^2/2-\sqrt{1-x^2} & \text{if\quad$|x|\leq 1$} \\
        \infty & \text{if\quad$|x|>1$}\\
  \end{cases}$
	\\[0.5ex]\hline
   
  Arctangent activation & $ \dfrac{2}{\pi}\mathrm{arctan}(x)$ & $
  \begin{cases}
		-\dfrac{2}{\pi}\ln(\cos(\dfrac{\pi x}{2}))-\dfrac{x^2}{2}& \text{if\quad$|x|<1$} \\
        \infty & \text{if\quad$|x|\geq 1$}\\
  \end{cases}$
	\\[0.5ex]\hline 
  
  \shortstack[c]{Hyperbolic tangent\\ activation}& $ \mathrm{tanh}(x)$ & $
  \begin{cases}
		x \mathrm{arctanh}(x) + \dfrac{\ln(1-x^2)-x^2}{2} & \text{if\quad$|x|<1$} \\
        \infty & \text{if\quad$|x|\geq 1$}\\
  \end{cases}$
	\\[0.5ex]\hline 
  
  Elliot activation& $\dfrac{x}{|x|+1}$ & $
  \begin{cases}
		-|x|-\ln(1-|x|)-\dfrac{x^2}{2} & \text{if\quad$|x|<1$} \\
        \infty & \text{if\quad$|x|\geq 1$}\\
  \end{cases}$
	\\[0.5ex]
	\hline 
  
  \shortstack[c]{Inverse hyperbolic\\ sine}& $\mathrm{arcsinh}(x)$ & $\mathrm{cosh}(x)-\dfrac{|x|^2}{2}$
	\\[0.5ex]\hline
    
  Logarithmic activation& $\mathrm{sgn}(x)\ln(|x|+1)$ & $\exp(|x|)-|x|-1-\dfrac{|x|^2}{2}$\\\hline 
    
    \end{tabular}
		
		\caption{Stable activation functions and their corresponding proximal mappings, see \cite{CP2018}.}
   \label{prox:act}
\end{center}
}
\end{table}
}

\subsection*{Acknowledgments} 
Many thanks to the reviewer for the constructive suggestions in relation with Lipschitz networks
and Plug-and-Play algorithms.
Further, we like to thank J.-C. Pesquet for pointing us to \cite{Combettes2018}, which we were not aware of when writing this paper.
Funding by the German Research Foundation (DFG) within the project STE 571/13-1, the RTG 1932 and the RTG 2088 is gratefully acknowledged.
We gratefully acknowledge the support of NVIDIA Corporation with the donation of the Quadro M5000 GPU used for this research.

\bibliographystyle{abbrv}
\bibliography{references}

\begin{thebibliography}{10}

\bibitem{AMS08}
P.-A. Absil, R.~Mahony, and R.~Sepulchre.
\newblock {\em Optimization Algorithms on Matrix Manifolds}.
\newblock Princeton and Oxford, Princeton University Press, 2008.

\bibitem{ALG19}
C.~Anil, J.~Lucas, and R.~Grosse.
\newblock Sorting out {L}ipschitz function approximation.
\newblock In K.~Chaudhuri and R.~Salakhutdinov, editors, {\em Proceedings of
  the 36th International Conference on Machine Learning}, volume~97 of {\em
  Proceedings of Machine Learning Research}, pages 291--301, Long Beach,
  California, USA, 2019. PMLR.

\bibitem{ASB2016}
M.~Arjovsky, A.~Shah, and Y.~Bengio.
\newblock Unitary evolution recurrent neural networks.
\newblock In {\em International Conference on Machine Learning}, pages
  1120--1128, 2016.

\bibitem{bansal2018can}
N.~Bansal, X.~Chen, and Z.~Wang.
\newblock Can we gain more from orthogonality regularizations in training deep
  networks?
\newblock In {\em Advances in Neural Information Processing Systems}, pages
  4261--4271, 2018.

\bibitem{BC11}
H.~H. Bauschke and P.~L. Combettes.
\newblock {\em Convex Analysis and Monotone Operator Theory in Hilbert Spaces}.
\newblock Springer, New York, 2011.

\bibitem{Beck17}
A.~Beck.
\newblock {\em First-Order Methods in Optimization}, volume~25 of {\em MOS-SIAM
  Series on Optimization}.
\newblock SIAM, 2017.

\bibitem{Bertsekas2011}
D.~P. Bertsekas.
\newblock Incremental proximal methods for large scale convex optimization.
\newblock {\em Mathematical programming}, 129:163--195, 2011.

\bibitem{BSS2017}
M.~Burger, A.~Sawatzky, and G.~Steidl.
\newblock First order algorithms in variational image processing.
\newblock In {\em Operator Splittings and Alternating Direction Methods}.
  Springer, 2017.

\bibitem{CWE2016}
S.~H. Chan, X.~Wang, and O.~A. Elgendy.
\newblock Plug-and-play {ADMM} for image restoration: Fixed-point convergence
  and applications.
\newblock {\em IEEE Transactions on Computational Imaging}, 3:84--98, 2016.

\bibitem{CPR2013}
E.~Chouzenoux, J.-C. Pesquet, and A.~Repetti.
\newblock Variable metric forward-backward algorithm for minimizing the sum of
  a differentiable function and a convex function.
\newblock {\em Journal of Optimization Theory and Applications}, 162:107--132,
  2014.

\bibitem{CB2016}
O.~Christensen.
\newblock {\em An Introduction to Frames and Riesz Bases}.
\newblock Springer, 2016.

\bibitem{Combettes2018}
P.~L. Combettes.
\newblock Monotone operator theory in convex optimization.
\newblock {\em Mathematical Programming}, 170(1):177--206, 2018.

\bibitem{CP2018}
P.~L. Combettes and J.-C. Pesquet.
\newblock Deep neural network structures solving variational inequalities.
\newblock {\em Set-Valued and Variational Analysis}, pages 1--28, 2020.

\bibitem{CW05}
P.~L. Combettes and V.~R. Wajs.
\newblock Signal recovery by proximal forward-backward splitting.
\newblock {\em Multiscale Modeling $\&$ Simulation}, 4:1168--1200, 2005.

\bibitem{CV98}
Z.~Cvetkovi\'c and M.~Vetterli.
\newblock Oversampled filter banks.
\newblock {\em IEEE Transactions on Signal Processing}, 46:1245--1255, 1998.

\bibitem{DDM04}
I.~Daubechies, M.~Defrise, and C.~De~Mol.
\newblock An iterative thresholding algorithm for linear inverse problems with
  a sparsity constraint.
\newblock {\em Communications on Pure and Applied Mathematics: A Journal Issued
  by the Courant Institute of Mathematical Sciences}, 57(11):1413--1457, 2004.

\bibitem{Dorobantu2016}
V.~Dorobantu, P.~A. Stromhaug, and J.~Renteria.
\newblock {DIZZYRNN}: Reparameterizing recurrent neural networks for
  norm-preserving backpropagation.
\newblock In {\em CoRR abs/1612.04035}, 2016.

\bibitem{EA2006}
M.~Elad and M.~Aharon.
\newblock Image denoising via sparse and redundant representations over learned
  dictionaries.
\newblock {\em IEEE Transactions on Image processing}, 15(12):3736--3745, 2006.

\bibitem{FMMC2017}
T.~Frerix, T.~M\"ollenhoff, M.~Moeller, and D.~Cremers.
\newblock Proximal backpropagation.
\newblock Technical report, ArXiv Preprint arXiv:1706.04638, 2017.

\bibitem{GP2019}
J.~A. Geppert and G.~Plonka.
\newblock Frame soft shrinkage operators are proximity operators.
\newblock Technical report, arXiv preprint arXiv:1910.01820, 2019.

\bibitem{GL2013}
G.~H. Golub and C.~F.~V. Loan.
\newblock {\em Matrix Computations}.
\newblock The Johns Hopkins University Press, 2013.

\bibitem{GSS2015}
J.~Goodfellow, J.~Shlens, and C.~Szegedy.
\newblock Expalining and harnessing adversarial examples.
\newblock In {\em International Conference on Learning Representations}, 2015.

\bibitem{GFPC18}
H.~Gouk, E.~Frank, B.~Pfahringer, and M.~Cree.
\newblock Regularisation of neural networks by enforcing {L}ipschitz
  continuity.
\newblock {\em arXiv preprint arXiv:1804.04368}, 2018.

\bibitem{GN2018}
R.~Gribonval and M.~Nikolova.
\newblock A characterization of proximity operators.
\newblock {\em arXiv preprint arXiv:1807.04014}, 2018.

\bibitem{HF2016}
M.~Harandi and B.~Fernando.
\newblock Generalized backpropagation, etude de cas: Orthogonality.
\newblock In {\em CoRR abs/1611.05927}, 2016.

\bibitem{huang2018orthogonal}
L.~Huang, X.~Liu, B.~Lang, A.~W. Yu, Y.~Wang, and B.~Li.
\newblock Orthogonal weight normalization: Solution to optimization over
  multiple dependent stiefel manifolds in deep neural networks.
\newblock In {\em Thirty-Second AAAI Conference on Artificial Intelligence},
  2018.

\bibitem{HCC2019}
T.~P. Huster, C.-Y.~J. Chiang, and R.~Chadha.
\newblock Limitations of the {L}ipschitz constant as a defense against
  adversarial examples.
\newblock In {\em ECML PKDD 2018 Workshops}, pages 16--29. Springer
  International Publishing, 2019.

\bibitem{jing2017tunable}
L.~Jing, Y.~Shen, T.~Dubcek, J.~Peurifoy, S.~Skirlo, Y.~LeCun, M.~Tegmark, and
  M.~Solja{\v{c}}i{\'c}.
\newblock Tunable efficient unitary neural networks ({EUNN}) and their
  application to {RNN}s.
\newblock In {\em Proceedings of the 34th International Conference on Machine
  Learning-Volume 70}, pages 1733--1741. JMLR. org, 2017.

\bibitem{KKHP2017}
E.~Kobler, T.~Klatzer, K.~Hammernik, and T.~Pock.
\newblock Variational networks: Connecting variational methods and deep
  learning.
\newblock In {\em German conference on pattern recognition}, pages 281--293.
  Springer, 2017.

\bibitem{LM2018}
G.~Lerman and T.~Maunu.
\newblock An overview of robust subspace recovery.
\newblock {\em Proceedings of the {IEEE}}, 106(8):1380--1410, 2018.

\bibitem{lezcano2019cheap}
M.~Lezcano-Casado and D.~Mart{\'\i}nez-Rubio.
\newblock Cheap orthogonal constraints in neural networks: A simple
  parametrization of the orthogonal and unitary group.
\newblock {\em arXiv preprint arXiv:1901.08428}, 2019.

\bibitem{mallat2008}
S.~Mallat.
\newblock {\em A wavelet tour of signal processing: the sparse way}.
\newblock Access Online via Elsevier, 2008.

\bibitem{MKKY2018}
T.~Miyato, T.~Kataoka, M.~Koyama, and Y.~Yoshida.
\newblock Spectral normalization for generative adversarial networks.
\newblock In {\em International Conference on Learning Representations}, 2018.

\bibitem{Moreau65}
J.-J. Moreau.
\newblock Proximit\'{e} et dualit\'{e} dans un espace {H}ilbertien.
\newblock {\em Bulletin de la Soci\'{e}t\'{e} Math\'{e}matique de France},
  93:273--299, 1965.

\bibitem{NNSS2019}
S.~Neumayer, M.~Nimmer, S.~Setzer, and G.~Steidl.
\newblock On the rotational invariant $l_1$-norm {PCA}.
\newblock {\em Linear Algebra and its Applications}, 587:243--270, 2019.

\bibitem{NA2005}
Y.~Nishimori and S.~Akaho.
\newblock Learning algorithms utilizing quasi-geodesic flows on the {S}tiefel
  manifold.
\newblock {\em Neurocomputing}, 67:106--135, 2005.

\bibitem{PS2006}
G.~Plonka and G.~Steidl.
\newblock A multiscale wavelet-inspired scheme for nonlinear diffusion.
\newblock {\em International Journal of Wavelets, Multiresolution and
  Information Processing}, 4(1):1--21, 2006.

\bibitem{Reich1979}
S.~Reich.
\newblock Weak convergence theorems for nonexpansive mappings in {B}anach
  spaces.
\newblock {\em Journal of Mathematical Analysis and Applications}, 67:274--276,
  1979.

\bibitem{SGL2019}
H.~Sedghi, V.~Gupta, and P.~M. Long.
\newblock The singular values of convolutional layers.
\newblock In {\em International Conference on Learning Representations}, 2019.

\bibitem{Se09}
S.~Setzer.
\newblock Operator splittings, {B}regman methods and frame shrinkage in image
  processing.
\newblock {\em International Journal of Computer Vision}, 92(3):265--280, 2011.

\bibitem{SKM2019}
H.~Sommerhoff, A.~Kolb, and M.~Moeller.
\newblock Energy dissipation with plug-and-play priors.
\newblock In {\em NeurIPS 2019 Workshop}, 2019.

\bibitem{SVW2016}
S.~Sreehariand, S.~V. Venkatakrishnan, and B.~Wohlberg.
\newblock Plug-and-play priors for bright field electron tomography and sparse
  interpolation.
\newblock {\em IEEE Transactions on Computational Imaging}, 2:408--423, 2016.

\bibitem{SWPMW2004}
G.~Steidl, J.~Weickert, T.~Brox, P.~Mr\'azek, and M.~Welk.
\newblock On the equivalence of soft wavelet shrinkage, total variation
  diffusion, total variation regularization, and sides.
\newblock {\em SIAM Journal on Numerical Analysis}, 42(2):686--713, 2004.

\bibitem{Sun2018AnOP}
Y.~Sun, B.~Wohlberg, and U.~Kamilov.
\newblock An online plug-and-play algorithm for regularized image
  reconstruction.
\newblock {\em IEEE Transactions on Computational Imaging}, 5:395--408, 2018.

\bibitem{TBF2018}
A.~M. Teodoro, J.~M. Bioucas-Dias, and M.~A. Figueiredo.
\newblock A convergent image fusion algorithm using scene-adapted
  {G}aussian-mixture-based denoising.
\newblock {\em IEEE Transactions on Image Processing}, 28(1):451--463, 2018.

\bibitem{TSS2018}
Y.~Tsuzuku, I.~Sato, and M.~Sugiyama.
\newblock Lipschitz-margin training: Scalable certification of perturbation
  invariance for deep neural networks.
\newblock In {\em Advances in Neural Information Processing Systems 31}, pages
  6541--6550. 2018.

\bibitem{Vorontsov2017}
E.~Vorontsov, C.~Trabelsi, S.~Kadoury, and C.~Pal.
\newblock On orthogonality and learning recurrent networks with long term
  dependencies.
\newblock In {\em Proceedings of the 34th International Conference on Machine
  Learning-Volume 70}, pages 3570--3578. JMLR. org, 2017.

\bibitem{WY2013}
Z.~Wen and W.~Yin.
\newblock A feasible method for optimization with orthogonality constraints.
\newblock {\em Mathematical Programming}, 142(1--2):397--434, 2013.

\bibitem{Wisdom2016}
S.~Wisdom, T.~Powers, J.~Hershey, J.~Le~Roux, and L.~Atlas.
\newblock Full-capacity unitary recurrent neural networks.
\newblock In {\em Advances in neural information processing systems}, pages
  4880--4888, 2016.

\end{thebibliography}
\end{document}